\documentclass[12pt]{article}
\usepackage{amsfonts}
\usepackage{amssymb}
\usepackage{mathrsfs}
\usepackage{srcltx}
\usepackage{amsmath}
\usepackage{amsthm}
\usepackage{amstext}
\usepackage{amsopn}
\usepackage{graphicx}

\newtheorem{theorem}{Theorem}[section]
\newtheorem{lemma}[theorem]{Lemma}

\theoremstyle{definition}

\theoremstyle{remark}
\newtheorem{remark}[theorem]{Remark}

\numberwithin{equation}{section}

\newcommand{\ba}{\begin{array}}
\newcommand{\ea}{\end{array}}
\newcommand{\f}{\frac}

\newcommand{\la}{\lambda}

\newcommand{\ds}{\displaystyle}

\begin{document}
\date{}
\title{ \bf\large{Stationary patterns of a diffusive predator-prey model with Crowley-Martin functional response}\footnote{This research is supported by the National Natural Science Foundation of China (Nos. 11371111 and 11301111)}}
 \author{Shanshan Chen\footnote{Email: chenss@hit.edu.cn},\ \  Junjie Wei\footnote{Corresponding Author, Email: weijj@hit.edu.cn}, and Jinzhu Yu
 \\
{\small  Department of Mathematics,  Harbin Institute of Technology,\hfill{\ }}\\
 {\small Weihai, Shandong, 264209, P.R.China.\hfill{\ }}}
\maketitle

\begin{abstract}
{A diffusive predator-prey system with predator interference and Neumann boundary conditions
is considered in this paper. We derive
some results on the existence and nonexistence of nonconstant stationary solutions.
It is shown that there exist no nonconstant stationary solutions when the effect of the predator interference is strong or the conversion rate of the predator is large, and nonconstant stationary solutions emerge when the diffusion rate of the predator is large.}

\end{abstract}

\section{Introduction}

The interaction between the predator and prey is closely related with
the functional response of the predator, which refers to
the per capita feeding rate of the predator upon its prey \cite{Berec,Murray}.
In general, a diffusive predator-prey model takes the form \cite{Murdoch}
\begin{equation}\label{1.1}
\begin{cases}
  \ds\frac{\partial u}{\partial t}=d_1\Delta u+ru\left(1-\f{u}{k}\right)-bp(u,v)v, & x\in \Omega,\; t>0,\\
 \ds\frac{\partial v}{\partial t}=d_2\Delta v-d v+cp(u,v)v, & x\in\Omega,\; t>0,\\
\end{cases}
\end{equation}
where $u(x,t)$ and $v(x,t)$ are the densities
of the prey and predator at time $t$ and location $x$ respectively, $d_1$, $d_2$, $r$, $d$, $k$, $b$ and $c$ are positive constants, and $p(u,v)$ represents the functional response of the predator.
If $p(u,v)$ depends only on $u$, then it is referred to as the predator density-independent functional response.
The predator density-independent functional responses are generally
classified into four Holling types: type I-IV\cite{Holling}. When $p(u,v)$ is Holling type I functional response, that is,
\begin{equation}\label{H1}
p(u,v)=\begin{cases}
  u, & u< 1/\alpha,\\
 1/\alpha, & u\ge 1/\alpha,
\end{cases}
\end{equation}
Seo and Kot \cite{Seo}
found that the kinetic system of model \eqref{1.1} possesses two limit cycles and these cycles arise through global cyclic-fold bifurcations. When $p(u,v)$ is the following Holling type II functional response
\begin{equation}\label{H2}
p(u,v)=\ds\f{u}{1+\alpha u},
\end{equation}
the ODE system of model \eqref{1.1} has been investigated extensively on the aspect of the global stability and existence and uniqueness of a limit cycle \cite{Cheng,Hsu1,Hsu2,Hsu3}. We refer to \cite{Peng-Shi,Wang,Yi} on the bifurcations of steady states and periodic solutions and the existence and nonexistence of nonconstant steady states for PDE system with homogeneous Neumann boundary conditions. For PDE system subject to homogeneous Dirichlet boundary conditions, Zhou and Mu \cite{muzhou} gave the necessary and sufficient condition for the existence of positive steady states of system \eqref{1.1}. Moreover,
other predator-prey models with Holling type II functional response were studied in \cite{Chen-yu2,Du-Lou3,Du-Lou,Du-Shi1,Du-Shi}.
When $p(u,v)$ is Holling type III or IV functional response, the dynamics and spatiotemporal patterns of system \eqref{1.1} were investigated in \cite{pangw,RuanX1,WangJF,RuanX2,ZhuC} and references therein.

The above mentioned Holling type functional responses can induce different dynamical behaviors and spatiotemporal patterns, which can be used to explain the ecological complexity. However, these functional responses are all independent of the predator density, which implies that the competition among
predators for food occurs only in the process of prey depletion \cite{Berec}. This is not realistic sometimes and the predator interference was investigated by many researchers. For example, when $p(u,v)$ is Holling type I functional response with predator interference, that is,
\begin{equation}\label{H1i}
p(u,v)=\begin{cases}
  \ds\f{u}{1+\beta v}, & u< 1/\alpha,\\
 \ds\f{1}{\alpha(1+\beta v)}, & u\ge 1/\alpha,
\end{cases}
\end{equation}
Seo and DeAngelis \cite{Seo2} studied the stability and bifurcations of equilibria for the kinetic system of model \eqref{1.1}. Here $\beta v$ models the mutual interference among predators, and if $\beta=0$, Eq. \eqref{H1i} is reduced to Holling type I functional response. Similarly, the following functional response can be derived from Holling type II functional response:
\begin{equation}\label{H2i}
p(u,v)=\ds\f{u}{1+\alpha u+\beta v}.
\end{equation}
This functional response is always referred to as the Beddington-DeAngelis (BD) functional response, which was introduced by
Beddington \cite{Beddington} and DeAngelis et al. \cite{DeAngelis}. The dynamics of model \eqref{1.1} with BD functional response was investigated in \cite{Cantrell,Seo2,ZhangJ}.
Similarly, the following functional response, proposed by Bazykin \cite{Bazykin} and Crowley and Martin \cite{Crowley},
\begin{equation}\label{H2c}
p(u,v)=\ds\f{u}{(1+\alpha u)(1+\beta v)}
\end{equation}
also models the predator interference, which is referred to as the Crowley-Martin (CM) functional response. For this functional response, Sambath et al. \cite{Sambath}
studied the stability and bifurcations of the positive equilibrium of system \eqref{1.1} when the positive equilibrium is unique. Wang and Wu \cite{WangM} studied a slightly different model, where the growth rate of the predator is logistic type in the absence of prey, and obtained the stability and multiplicity of the positive
solutions when some parameters are large or small.
We remark that there are also many results on other predator-prey models with CM functional response \cite{LiW,ShiR,WeiW,ZhouJ2,ZhouJ1}.

In this paper, we revisit model \eqref{1.1} with CM functional response and no
flux boundary conditions, that is,
\begin{equation}\label{CM1}
\begin{cases}
  \ds\frac{\partial u}{\partial t}-d_1\Delta u=ru\left(1-\ds\f{u}{k}\right)-\ds\frac{buv}{(1+\alpha u)(1+\beta v)}, & x\in \Omega,\; t>0,\\
 \ds\frac{\partial v}{\partial t}-d_2\Delta v=-dv+\ds\frac{cuv}{(1+\alpha u)(1+\beta v)}, & x\in\Omega,\; t>0,\\
 \partial_\nu  u=\partial_\nu
  v=0,& x\in \partial \Omega,\;
 t>0,\\
 u(x,0)=u_0(x)\ge(\not\equiv)0, \;\; v(x,0)=v_0(x)\ge(\not\equiv)0,& x\in\Omega,
\end{cases}
\end{equation}
where $\Omega$ is a bounded domain in $\mathbb{R}^N$ ($N\le 3$) with a
smooth boundary $\partial \Omega$; $u(x,t)$ and $v(x,t)$ stand for the densities
of the prey and predator at time $t$ and location
$x$ respectively; $r>0$ is the intrinsic growth rate of the prey; $k>0$ is the carrying capacity of the prey;
$d>0$ is the mortality rate of the predator; $b,c>0$ measure the
interaction strength between the predator and prey; $\alpha>0$ measures the prey's ability to evade attack, and $\beta>0$ measures the mutual interference between predators.
By using the following rescaling,
$$\tilde t=rt,\;\;\tilde u=\ds\f{u}{k},\;\;\tilde v=\ds\f{bv}{r},\;\tilde \alpha= \alpha k\;\;\tilde \beta=\ds\f{\beta r}{b},\;\;\tilde d=\ds\f{d}{r},\;\;\tilde c=\ds\f{ck}{r},\;\;\tilde d_1=\ds\f{d_1}{r},\;\;\tilde d_2=\ds\f{d_2}{r},$$
and dropping the tilde sign,
system \eqref{CM1} can be simplified as follows:
\begin{equation}\label{CM}
\begin{cases}
  \ds\frac{\partial u}{\partial t}-d_1\Delta u=u\left(1-u\right)-\ds\frac{uv}{(1+\alpha u)(1+\beta v)}, & x\in \Omega,\; t>0,\\
 \ds\frac{\partial v}{\partial t}-d_2\Delta v=-dv+\ds\frac{cuv}{(1+\alpha u)(1+\beta v)}, & x\in\Omega,\; t>0,\\
 \partial_\nu  u=\partial_\nu
  v=0,& x\in \partial \Omega,\;
 t>0,\\
 u(x,0)=u_0(x)\ge(\not\equiv)0, \;\; v(x,0)=v_0(x)\ge(\not\equiv)0,& x\in\Omega.
\end{cases}
\end{equation}
Here parameter $c$ represents the conversion rate of the predator, all the parameters are positive, and $\Omega$ is a bounded domain in $\mathbb{R}^N$ ($N\le 3$) with a
smooth boundary $\partial \Omega$.
The results in \cite{Sambath} are mainly derived under certain conditions where system \eqref{CM} has a unique constant positive
equilibrium. However, there are two or three constant positive equilibria of system \eqref{CM} under certain conditions.
The main purpose of this paper is to understand the stationary solutions even when system \eqref{CM} has more than one constant positive equilibrium. The rest of the paper is organized as follows. In Section 2, we study the existence and global stability of constant positive equilibria of system \eqref{CM}.
In Section 3, we establish some existence and nonexistence results on nonconstant steady states of system \eqref{CM}.
 Throughout this paper, $\mathbb{N}_0 =\mathbb{N}\cup \{0\}$, where $\mathbb{N}$ is the set of natural numbers, and \begin{equation}\label{mmu}0=\mu_0<\mu_1<\mu_2<\cdots<\mu_j<\cdots
 \end{equation}
 are the eigenvalues of operator $-\Delta$ in $\Omega$ with the homogeneous Neumann boundary condition.

\section {Equilibria and stability}
In this section, we consider the existence and stability of constant positive equilibria of system \eqref{CM}. One can easily check that $(u,v)$ is a constant positive equilibrium of system \eqref{CM} if and only if $u\in(0,1)$ is a solution of the following equation
\begin{equation}\label{uvs}
\ds\f{d(1+\alpha u)}{\beta c u}=G(u),
\end{equation}
where
\begin{equation}\label{gu}
G(u)=\ds\f{1}{\beta}-(1-u)(1+\alpha u).
\end{equation}
In the following, we will give two lemmas on the relations between parameter $c$ and the solution $u$ of Eq. \eqref{uvs}. In fact, parameter $c$ can be regarded as a function of $u$, defined by
\begin{equation}\label{cu}
C(u)=\ds\f{d(1+\alpha u)}{\beta uG(u)}.
\end{equation}
Noticing that $c$ is positive, we see that the domain of $C(u)$ is
\begin{equation}\label{dcu}
\mathcal {D}\left(C(u)\right)=\{u\in(0,1): G(u)>0\}.
\end{equation}
We first consider the case of $\alpha\le 1$, where $G(u)$ is strictly increasing.
\begin{lemma}\label{l1}
Assume that $\alpha\le1$.
\begin{enumerate}
\item [(i)] If $\beta\le1$, then $\mathcal {D}\left(C(u)\right)=(0,1)$, where $\mathcal {D}\left(C(u)\right)$ is defined as in Eq. \eqref{dcu}, and $C'(u)<0$. Moreover,
    $$\lim_{u\to 0^+} C(u)=\infty, \;\;\text{and}\;\;\lim_{u\to 1^-} C(u)=d(1+\alpha).$$
\item [(ii)] If $\beta>1$, then $G(u)$ has a unique positive zero $u_*$, $\mathcal {D}\left(C(u)\right)=(u_*,1)$, and $C'(u)<0$ for $u\in(u_*,1)$. Moreover,
    $$\lim_{u\to u_*^+} C(u)=\infty, \;\;\text{and}\;\;\lim_{u\to 1^-} C(u)=d(1+\alpha).$$
\end{enumerate}
\end{lemma}
\begin{proof}
We only prove part $(ii)$, and part $(i)$ can be proved similarly. Since $\beta >1$, we see that $G(u)$
has a unique positive zero $u_*\in(0,1)$, and $G(u)>0$ if and only if $u\in(u_*,1)$, which leads to $\mathcal {D}\left(C(u)\right)=(u_*,1)$. Direct computation yields $G'(u)>0$ and $\ds\left[\f{d(1+\alpha u)}{\beta u}\right]'<0$ for $u\in(u_*,1)$, and hence $C'(u)<0$ for $u\in(u_*,1)$.
\end{proof}
Then we consider the case of $\alpha>1$, which is more complicated than the above case (see Fig. \ref{fig1} for the sketch maps of function $C(u)$ under different conditions).
\begin{figure}[htbp]
\centering\includegraphics[width=0.5\textwidth]{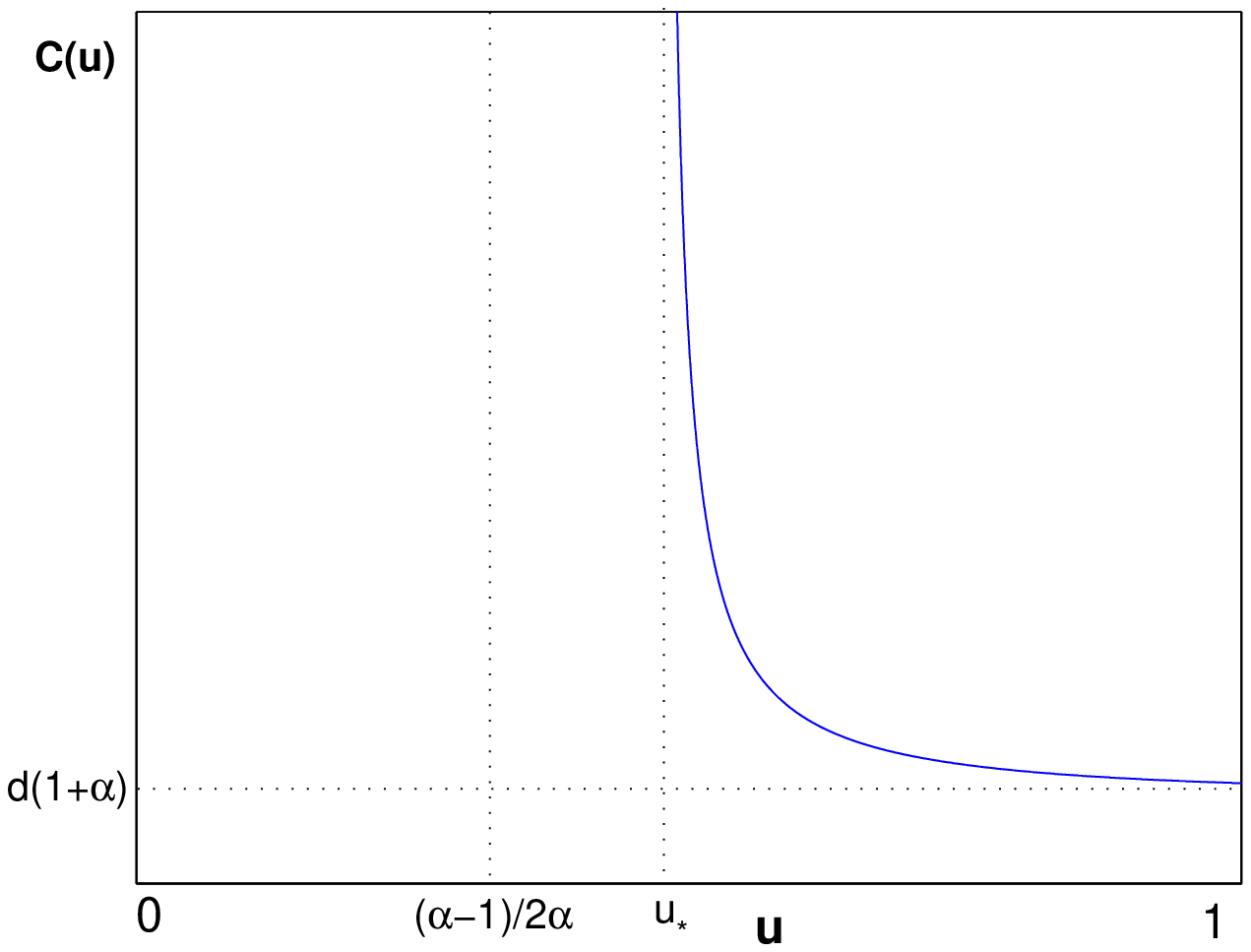}\includegraphics[width=0.5\textwidth]{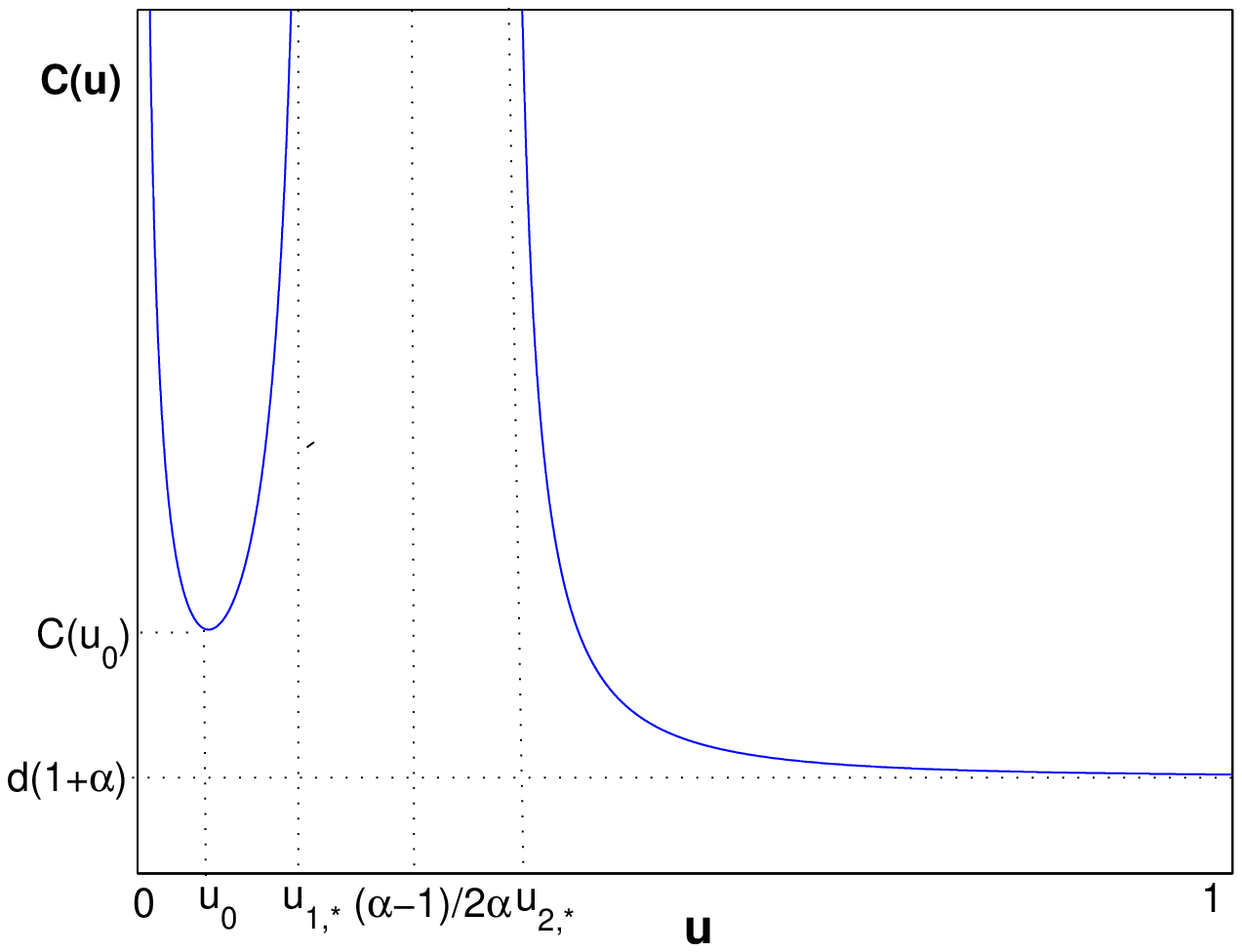}
\centering\includegraphics[width=0.5\textwidth]{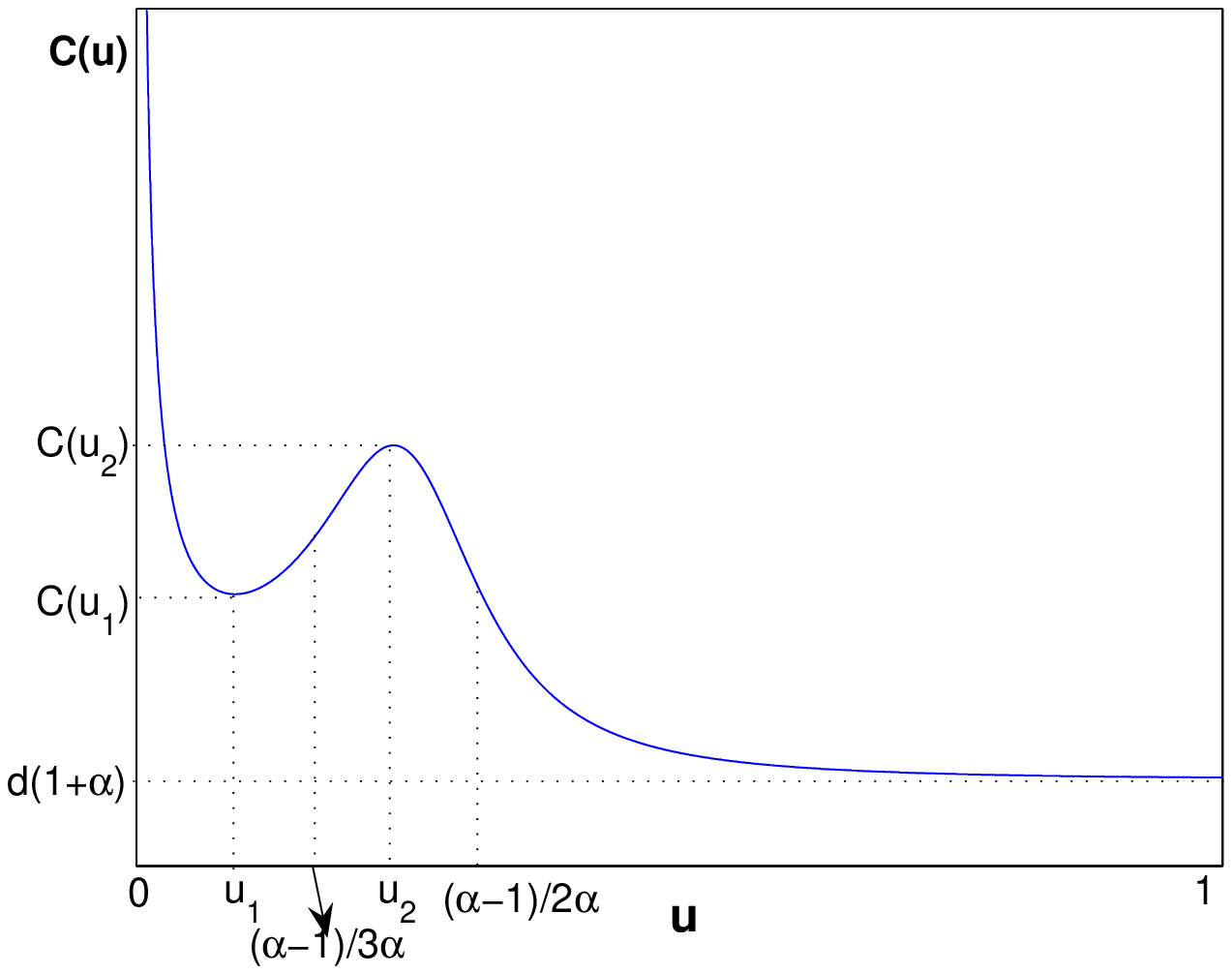}\includegraphics[width=0.5\textwidth]{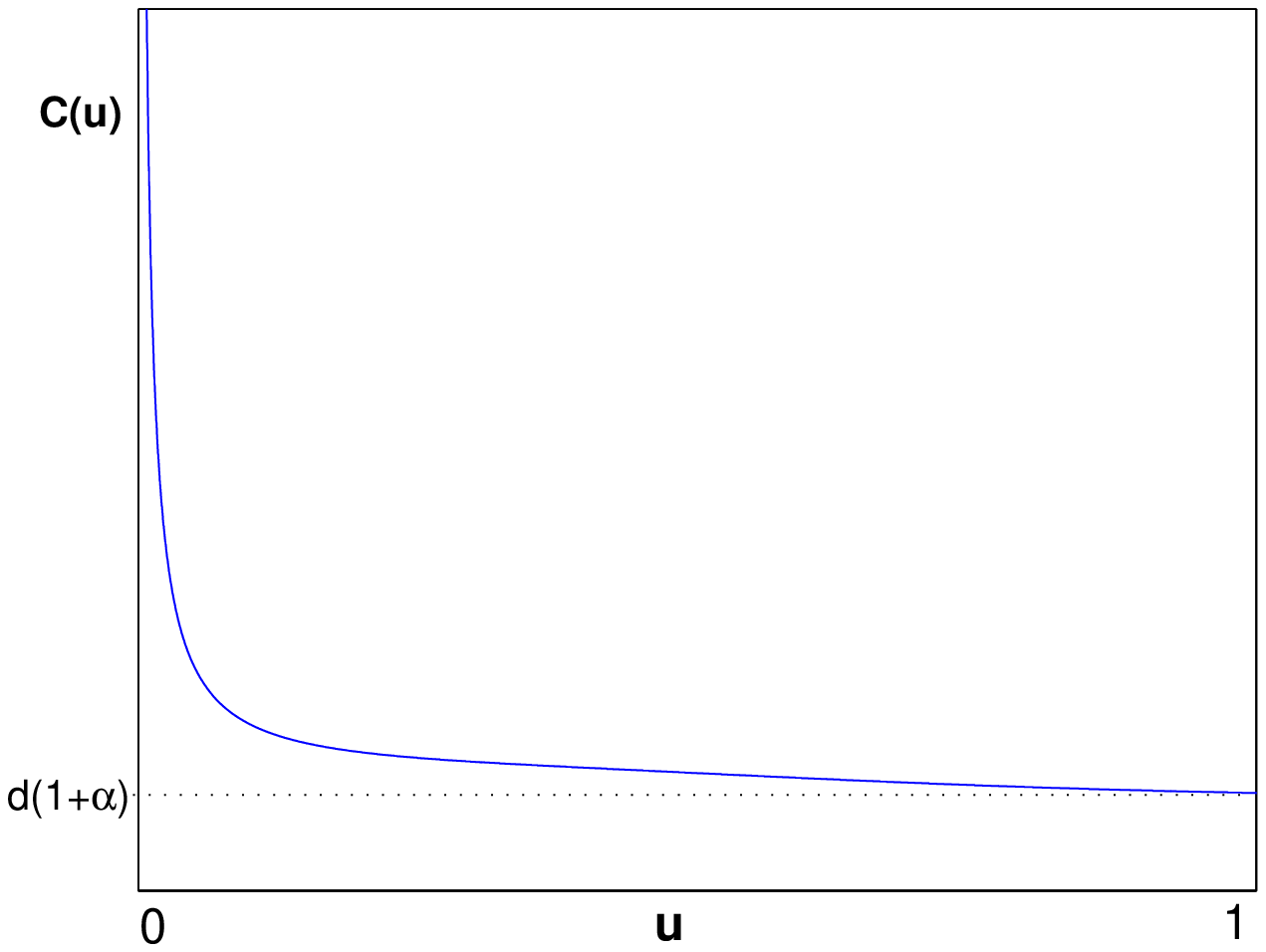}
\caption{Sketch maps of $C(u)$ for $\alpha>1$. (Upper
left) $\beta\ge1$; (Upper right) $\ds\f{4\alpha}{(1+\alpha)^2}\le\beta<1$; (Lower left) $\gamma(\alpha)\le \beta<\ds\f{4\alpha}{(1+\alpha)^2}$; (Lower right) $0<\beta<\gamma(\alpha)$.
  \label{fig1}}
\end{figure}
\begin{lemma}\label{l2}
Assume that $\alpha>1$.
\begin{enumerate}
\item [(i)] If $\beta\ge1$, then $G(u)$ has a unique positive zero $u_*$ satisfying $u_*\ge\ds\f{\alpha-1}{\alpha}$, and $C(u)$ satisfies the following properties.
    \begin{enumerate}
    \item [($i_1$)] $\mathcal {D}\left(C(u)\right)=(u_*,1)$, and $C'(u)<0$ for $u\in(u_*,1)$.
    \item [($i_2$)] $\lim_{u\to u_*^+} C(u)=\infty$, and $\lim_{u\to 1^-} C(u)=d(1+\alpha)$.
    \end{enumerate}
\item [(ii)] If $\ds\f{4\alpha}{(1+\alpha)^2}\le\beta<1$, then $G(u)$ has two positive zeros $u_{1,*}$ and $u_{2,*}$ satisfying $u_{1,*}<\ds\f{\alpha-1}{2\alpha}< u_{2,*}$ for $\beta>\ds\f{4\alpha}{(1+\alpha)^2}$ and $u_{1,*}= u_{2,*}=\ds\f{\alpha-1}{2\alpha}$ for $\beta=\ds\f{4\alpha}{(1+\alpha)^2}$, and $C(u)$ satisfies the following properties.
    \begin{enumerate}
     \item [($ii_1$)]$\mathcal {D}\left(C(u)\right)=(0,u_{1,*})\cup(u_{2,*},1)$.
     \item [($ii_2$)] There exists $u_0\in(0,u_{1,*})$ such that $C'(u_0)=0$, $C'(u)<0$ for $u\in (0,u_0)\cup(u_{2,*},1)$, and $C'(u)>0$ for $u\in(u_0,u_{1,*})$.
   \item [($ii_3$)]$\lim_{u\to 0^+} C(u)=\lim_{u\to u^-_{1,*}} C(u)=\lim_{u\to u^+_{2,*}} C(u)=\infty$, and $\;\;\lim_{u\to 1^-} C(u)=d(1+\alpha)$.
    \end{enumerate}
\item [(iii)] If $\gamma(\alpha)\le \beta<\ds\f{4\alpha}{(1+\alpha)^2}$, where
\begin{equation}\label{ggm}
\gamma(\alpha)=\ds\f{27\alpha}{(\alpha-1)^2(\alpha+8)+27\alpha},
\end{equation}
then $G(u)$ has no zeros, and $C(u)$ satisfies the following properties.
\begin{enumerate}
\item [($iii_1$)] $\mathcal {D}\left(C(u)\right)=(0,1)$.
\item [($iii_2$)] There exist $u_1$ and $u_2$ such that $C'(u_1)=C'(u_2)=0$, $C'(u)<0$ for $u\in (0,u_1)\cup(u_2,1)$, and $C'(u)>0$ for $u\in (u_1,u_2)$. Here $u_1$ and $u_2$ satisfy $u_1<\ds\f{\alpha-1}{3\alpha}< u_2<\ds\f{\alpha-1}{2\alpha}$ for $\beta >\gamma(\alpha)$ and $u_1=u_2=\ds\f{\alpha-1}{3\alpha}$ for $\beta=\gamma(\alpha)$.
\item [($iii_3$)]$\lim_{u\to 0^+} C(u)=\infty$, and $\lim_{u\to 1^-} C(u)=d(1+\alpha)$.
\end{enumerate}
\item [(iv)] If $0<\beta<\gamma(\alpha)$, where $\gamma(\alpha)$ is defined as in Eq. \eqref{ggm}, then
$G(u)$ has no zeros, and $C(u)$ satisfies the following properties.
\begin{enumerate}
\item [($iv_1$)]$\mathcal {D}\left(C(u)\right)=(0,1)$, and $C'(u)<0$ for $u\in(0,1)$.
\item [($iv_2$)]$\lim_{u\to 0^+} C(u)=\infty$, and $\lim_{u\to 1^-} C(u)=d(1+\alpha)$.
\end{enumerate}
\end{enumerate}
\end{lemma}
\begin{proof}
Similarly to the arguments in the proof of Lemma \ref{l1}, we can prove part $(i)$. In the following we only consider the case of $\beta<1$. One checks that
\begin{equation}\label{cud}
C'(u)=\ds\f{dH(u)}{\beta u^2G^2(u)},
\end{equation}
where
\begin{equation}\label{hu}
H(u)=-2\alpha^2u^3+\alpha(\alpha-4)u^2+2(\alpha-1)u-\left(\ds\f{1}{\beta}-1\right).
\end{equation}
Therefore, the sign of $C(u)$ is determined by $H(u)$. It is easily seen that
$$H(0)<0,\;\;H(1)<0,\;\;H\left(\ds\f{\alpha-1}{2\alpha}\right)=\ds\f{(\alpha+1)^2}{4\alpha}-\ds\f{1}{\beta},$$
$H'(u)$ has a unique positive zero $\ds\f{\alpha-1}{3\alpha}$, and $$H\left(\ds\f{\alpha-1}{3\alpha}\right)>(<)0\;\;\text{if}\;\;
\beta>(<)\gamma(\alpha).$$
When $\beta<\ds\f{4\alpha}{(1+\alpha)^2}$, $G(u)$ has no zeros, which implies that $\mathcal {D}\left(C(u)\right)=(0,1)$. From above analysis, we see that
\begin{enumerate}
\item [(1)] if $\gamma(\alpha)<\beta<\ds\f{4\alpha}{(1+\alpha)^2}$, then $C'(u)$ has two positive zeros $u_1$ and $u_2$ satisfying $u_1<\ds\f{\alpha-1}{3\alpha}< u_2<\ds\f{\alpha-1}{2\alpha}$, $C'(u)<0$ for $u\in (0,u_1)\cup(u_2,1)$, and $C'(u)>0$ for $u\in (u_1,u_2)$;
\item [(2)] if $\beta=\gamma(\alpha)$, then $C'(u)$ has a unique positive zero $\ds\f{\alpha-1}{3\alpha}$, and $C'(u)<0$ for $u\in \left(0,\ds\f{\alpha-1}{3\alpha}\right)\bigcup\left(\ds\f{\alpha-1}{3\alpha},1\right)$;
\item [(3)] if $\beta<\gamma(\alpha)$, then $C'(u)$ has no positive zeros, and $C'(u)<0$ for $u\in(0,1)$.
\end{enumerate}
Therefore, parts $(iii)$ and $(iv)$ are proved.

Finally, we consider the case of $\ds\f{4\alpha}{(1+\alpha)^2}<\beta<1$. Then $G(u)$ has two zeros $u_{1,*}$ and $u_{2,*}$ satisfying $u_{1,*}<\ds\f{\alpha-1}{2\alpha}<u_{2,*}$, which leads to
$\mathcal {D}\left(C(u)\right)=(0,u_{1,*})\cup(u_{2,*},1)$. Noticing that $G'(u)>0$ and $\left[\ds\f{d(1+\alpha u)}{\beta u}\right]'<0$ for $u\in(u_{2,*},1)$, we have
$C'(u)<0$ for $u\in(u_{2,*},1)$. Because $$H\left(\ds\f{\alpha-1}{2\alpha}\right)=\ds\f{(\alpha+1)^2}{4\alpha}-\ds\f{1}{\beta}>0,$$ and $\lim_{u\to 0^+} C(u)=\lim_{u\to u^-_{1,*}} C(u)=\infty$, we obtain that there exists $u_0\in(0,u_{1,*})$ such that $C'(u_0)=0$, $C'(u)<0$ for $u\in (0,u_0)$, and $C'(u)>0$ for $u\in(u_0,u_{1,*})$. Therefore, conclusion $(ii)$ is proved.
\end{proof}

\begin{remark}\label{re1}
We remark that if $\alpha>1$ and $\beta<1$, then
$$G(u)=\ds\f{1}{\beta}-1$$ has a unique positive root $\check u=\ds\f{\alpha-1}{\alpha}\in\mathcal {D}\left(C(u)\right)$ and $C'\left(u\right)<0$ for $u\in \left[\check u,1\right)$.
Moreover,
we have $G\left(\check u\right)>G(u)$ for any $u\in\left(0,\check u\right)$, which leads to
$C\left(\check u\right)<C(u)$ for any $u\in\left(0,\check u\right)$. Hence, if $\ds\f{4\alpha}{(1+\alpha)^2}\le\beta<1$ (respectively, $\gamma(\alpha)< \beta<\ds\f{4\alpha}{(1+\alpha)^2}$), then $C\left(u\right)<C(u_0)$ (respectively, $C\left( u\right)<C(u_1)$) for all $u\in \left[\check u,1\right)$, where $u_0$ and $u_1$ are defined as in Lemma \ref{l2}.
\end{remark}
Then,
by Lemmas \ref{l1} and \ref{l2} and Remark \ref{re1}, we can derive the following results on constant positive equilibria of system \eqref{CM}.
\begin{theorem}\label{equ}
System \eqref{CM} has no constant positive equilibrium for $c\in(0,d(1+\alpha)]$, and at least one constant positive equilibrium for $c\in(d(1+\alpha),\infty)$. Moreover,
\begin{enumerate}
\item[(i)] if $\alpha\le1$, or $\alpha >1$ but $\beta\in (0, \gamma(\alpha)]\cup[1,\infty)$ , where $\gamma(\alpha)$ is defined as in Eq. \eqref{ggm}, then system \eqref{CM} has a unique constant positive equilibrium for $c\in(d(1+\alpha),\infty)$.
\item[(ii)] if \begin{equation}
\alpha >1\;\;\text{and}\;\;\ds\f{4\alpha}{(1+\alpha)^2}\le\beta<1,
\end{equation}
then system \eqref{CM} has a unique constant positive equilibrium for $c\in(d(1+\alpha),c(u_0))$, two constant positive equilibria
for $c=c(u_0)$, and three constant positive equilibria for $c>c(u_0)$, where $u_0$ is defined as in Lemma \ref{l2}.
\item[(ii)] if \begin{equation}
\alpha >1\;\;\text{and}\;\;\gamma(\alpha)<\beta<\ds\f{4\alpha}{(1+\alpha)^2},
\end{equation}
then system \eqref{CM} has a unique constant positive equilibrium for $c\in(d(1+\alpha),c(u_1))\cup(c(u_2),\infty)$, two constant positive equilibria
for $c=c(u_1),c(u_2)$, and three constant positive equilibria for $c\in(c(u_1),c(u_2))$, where $u_1$ and $u_2$ are defined as in Lemma \ref{l2}.
\end{enumerate}
\end{theorem}
Then, we consider the stability of constant positive equilibria. For the simplicity of notations, we denote
\begin{equation}\label{pspi}
\begin{split}
&\phi_1(u)=\ds\f{u}{1+\alpha u},\;\;\phi_2(v)=\ds\f{v}{1+\beta v},\\
&\psi_1(u)=(1-u)(1+\alpha u),\;\;\psi_2(v)=-d(1+\beta v),
\end{split}
\end{equation}
and
\begin{equation}\label{gug}
G(\mathbf u)=\left(\begin{array}{c}
\phi_1(u)\left(\psi_1(u)-\phi_2(v)\right)\\
\phi_2(v)\left(\psi_2(v)+c\phi_1(u)\right)
\end{array}\right)\;\;\text{for}\;\;\mathbf u=(u,v).
\end{equation}
Let $\tilde {\mathbf u}=(\tilde u,\tilde v)$ be the positive equilibrium of system \eqref{CM}. Then, the stability of $\tilde {\mathbf u}$ is associated with the following eigenvalue problem
\begin{equation}\label{egen}
D\Delta \mathbf u+G_{\mathbf u}(\mathbf {\tilde u})\mathbf u=\mu  \mathbf u,
\end{equation}
where
\begin{equation}\label{cgu}
D=\left(\begin{array}{cc}
d_1&0\\
0&d_2
\end{array}\right),\;\;G_{\mathbf u}(\mathbf {\tilde u})=\left(\begin{array}{cc}
\phi_1(\tilde u)\psi'_1(\tilde u)&-\phi_1(\tilde u)\phi'_2(\tilde v)\\
c\phi_2(\tilde v)\phi'_1(\tilde u)&-d\beta \phi_2(\tilde v)
\end{array}\right),
\end{equation}
and $\tilde {\mathbf u}$ is locally asymptotically stable if all the eigenvalues of problem \eqref{egen} have negative real parts. We remark that $\mathbf{u}$ in Eq. \eqref{egen} should be replaced by $\mathbf{u}^T$ actually, and here we still use $\mathbf{u}$ for simplicity.
In fact, $\mu$ is an eigenvalue of problem \eqref{egen} if and only if $\mu$ is an eigenvalue of matrix $Q_j(\mathbf {\tilde u})=-\mu_jD+G_{\mathbf u}(\mathbf {\tilde u})$ for some $j\in\mathbb{N}_0$, where $\{\mu_j\}_{j=0}^\infty$ is defined as in Eq. \eqref{mmu}. Then, we obtain a sequence of characteristic equations \begin{equation}\label{char}
\la^2-\text{Tr } Q_j(\mathbf {\tilde u})\la+\text{Det } Q_j(\mathbf {\tilde u})=0,\;\;j\in \mathbb{N}_{0},\end{equation}
where
\begin{equation}\label{trd}
\begin{split}
&\text{Tr } Q_j(\mathbf {\tilde u})=-(d_1+d_2)\mu_j+\phi_1(\tilde u)\psi'_1(\tilde u)-d\beta \phi_2(\tilde v),\\
&\text{Det } Q_j(\mathbf {\tilde u})=d_1d_2\mu_j^2+\left(d\beta \phi_2(\tilde v)d_1-d_2\phi_1(\tilde u)\psi'_1(\tilde u)\right)\mu_j+\text{Det }G_{\mathbf u}(\mathbf {\tilde u}),\\
&\text{Det } G_{\mathbf u}(\mathbf {\tilde u})=\text{Det } Q_0(\mathbf {\tilde u})=\phi_1(\tilde u)\phi_2(\tilde v)\left(-d\beta\psi'_1(\tilde u)+c\phi'_1(\tilde u)\phi'_2(\tilde v)\right).
\end{split}
\end{equation}
Hence, $\tilde {\mathbf u}$ is locally asymptotically stable if
$\text{Tr }Q_j(\mathbf {\tilde u})<0$ and $\text{Det }Q_j(\mathbf {\tilde u})>0$
 for all $j\in\mathbb{N}_{0}$.
 To analyze the stability of constant positive equilibria of system \eqref{CM}, we first give the following result for further application.
 \begin{lemma}\label{sig}
Let $\tilde {\mathbf u}=(\tilde u,\tilde v)$ be a constant positive equilibrium of system \eqref{CM}. Then $\text{Det }G_{\mathbf u}(\tilde {\mathbf u})$ has the same sign as $-C'(\tilde u)$.
\end{lemma}
\begin{proof}
From Eqs. \eqref{uvs} and \eqref{gu}, we see that
$$C(u)G(u)=C(u)\left(\ds\f{1}{\beta}-\psi_1(u)\right)=\ds\f{d}{\beta \phi_1(u)},$$
where $G(u)$ is defined as in Eq. \eqref{gu}, and $\phi_1$ and $\psi_1$ are defined as in Eq. \eqref{pspi}.
Then
$$C'(\tilde u)G(\tilde u)=C(\tilde u)\psi'_1(\tilde u)-\ds\f{d\phi'_1(\tilde u)}{\beta\phi^2_1(\tilde u)}.$$
Noticing that $C(\tilde u)\phi_1(\tilde u)=d(1+\beta \tilde v)$, we have
$$C'(\tilde u)=\ds\f{(1+\beta \tilde v)}{\beta \phi_1(\tilde u)G(\tilde u)}\left[d\beta \psi'_1(\tilde u)-\ds\f{d\phi'_1(\tilde u)}{\phi_1(\tilde u)(1+\beta \tilde v)}\right].$$
An easy calculation implies that
$$\text{Det } G_{\mathbf u}(\mathbf {\tilde u})=\phi_1(\tilde u)\phi_2(\tilde v)\left[-d\beta \psi'_1(\tilde u)+\ds\f{d\phi'_1(\tilde u)}{\phi_1(\tilde u)(1+\beta \tilde v)}\right].$$
This completes the proof.
\end{proof}
Then, by virtue of Lemma \ref{sig}, we obtain some partial results on the stability of the constant positive equilibria of system \eqref{CM} in the following.
\begin{theorem} \label{los}
Let $\tilde {\mathbf u}=(\tilde u,\tilde v)$ be a constant positive equilibrium of system \eqref{CM}. If $\psi'_1(\tilde u)<0$ and $C'(\tilde u)<0$, then $\tilde {\mathbf u}$ is locally asymptotically stable.
\end{theorem}
\begin{proof}
Since $\psi'_1(\tilde u)<0$ and $C'(\tilde u)<0$, we see that $\text{Tr }Q_j(\mathbf {\tilde u})<0<0$ and $\text{Det }Q_j(\mathbf {\tilde u})>0$
 for all $j\in\mathbb{N}_{0}$. This complete the proof.
\end{proof}
Finally, under certain conditions, we derive the following results
on the global stability of the constant equilibrium.
\begin{theorem}\label{glo}
\begin{enumerate}
\item [(i)] Assume that \begin{equation}\label{cond0}
c<d(1+\alpha).
\end{equation}Then equilibrium $(1,0)$ is globally attractive.
\item [(ii)] Assume that one of the following is satisfied:
\begin{align}\label{cond1}
    & \alpha\le 1\;\;\text{and}\;\; c>d(1+\alpha),\\
    \label{cond2}
    & \alpha>1, \beta\ge1\;\;\text{and}\;\;  c>d(1+\alpha), \;\;\;\;\text{or}\;\;\\
    \label{cond3}
    & \alpha>1, \beta<1\;\;\text{and}\;\; c\in\left(d(1+\alpha),C\left(\f{\alpha-1}{\alpha}\right) \right].
\end{align}
Then system \eqref{CM} has a unique constant positive equilibrium $\tilde {\mathbf u}=(\tilde u,\tilde v)$, which is globally asymptotically stable.
\end{enumerate}
\end{theorem}
\begin{proof}
Part $(i)$ can be easily deduced by the comparison principle. Therefore, we omit the proof of part $(i)$ and just prove part $(ii)$.
By virtue of Remark \ref{re1} and Theorem \ref{equ}, we see that if one of Eqs. \eqref{cond1}-\eqref{cond3} is satisfied, then system \eqref{CM} has a unique constant positive equilibrium $\tilde {\mathbf u}=(\tilde u,\tilde v)$.
Set
\begin{equation}
V(u(x,t),v(x,t))=c\int_\Omega\int_{\tilde u}^u \ds\f{\phi_1(\xi)-\phi_1(\tilde u)}{\phi_1(\xi)}dx+\int_\Omega\int_{\tilde v}^v \ds\f{\phi_2(\eta)-\phi_2(\tilde v)}{\phi_2(\eta)}dx,
\end{equation}
where $\phi_1$ and $\phi_2$ are defined as in \eqref{pspi}.
Then
\begin{equation*}
\begin{split}
V_t(u(x,t),v(x,t))=&c\int_\Omega\ds\f{(u-\tilde u)\left(\psi_1(u)-\psi_1(\tilde u)\right)}{(1+\alpha u)(1+\alpha \tilde u)}dx-d\beta \int_\Omega \ds\f{(v-\tilde v)^2}{(1+\beta v)(1+\beta \tilde v)}dx\\
-&d_1c\phi_1(\tilde u)\int_\Omega\ds\f{\phi'_1(u)}{[\phi_1(u)]^2}|\nabla u|^2dx-d_2\phi_2(\tilde v)\int_\Omega\ds\f{\phi'_2(v)}{[\phi_2(v)]^2}|\nabla v|^2dx,
\end{split}
\end{equation*}
where $\psi_1$ is also defined as in \eqref{pspi}. If Eq. \eqref{cond1} is satisfied, then $\psi'_1(u)<0$ for $u>0$, which implies that
\begin{equation}\label{pse}
(u-\tilde u)\left(\psi_1(u)-\psi_1(\tilde u)\right)\le 0\;\;\text{for any}\;\;u> 0,
\end{equation}
and equality holds if and only if $u=\tilde u$.
By Remark \ref{re1} and Theorem \ref{equ}, we obtain that if Eq. \eqref{cond2} or \eqref{cond3} is satisfied, then $\tilde u\ge\ds\f{\alpha-1}{\alpha}$. This also leads to \eqref{pse}, and equality holds if and only if $u=\tilde u$. Then $\tilde {\mathbf u}$ is globally attractive.
A direct computation yields $\psi'_1(\tilde u)<0$ and $C'(\tilde u)<0$, which implies that $\tilde {\mathbf u}$ is locally asymptotically stable from Theorem \ref{los}. Therefore, $\tilde {\mathbf u}$ is globally asymptotically stable.
\end{proof}

\section{Stationary solutions}
In this section, we will investigate the steady states of system \eqref{CM}, which satisfy
\begin{equation}\label{stCM}
\begin{cases}
-d_1\Delta u=u\left(1-u\right)-\ds\f{uv}{(1+\alpha u)(1+\beta v)}, & x\in \Omega,\\
-d_2\Delta v=-dv+\ds\f{cuv}{(1+\alpha u)(1+\beta v)}, & x\in\Omega,\\
 \partial_\nu  u=\partial_\nu
  v=0,& x\in \partial \Omega,\\
\end{cases}
\end{equation}
and establish results on the existence and nonexistence of nonconstant positive steady states.
From above Theorem \ref{glo}, we see that if one of Eqs. \eqref{cond0}-\eqref{cond3} is satisfied, then all the solutions of system \eqref{CM}, regardless of the initial data,
converge to a constant steady state as time goes to infinity.
Therefore, we only need to consider the case that \begin{equation}\label{cas}
\alpha>1,\;\;\beta<1\;\;\text{and}\;\; c>C\left(\ds\f{\alpha-1}{\alpha}\right).\end{equation}
Throughout this section, we always assume that $\alpha>1$ and $\beta<1$ unless otherwise specified.
\subsection{The nonexistence}
In this subsection, we mainly study positive steady states of system \eqref{CM} when $c$ is large.
Suppose that $(u,v)$ satisfies Eq. \eqref{stCM}.
Let
$w=cu$, $z=v/c$ and $\rho=1/c$. Then $(w,v)$ satisfies
\begin{equation}\label{wv}
\begin{cases}
-d_1\Delta w=w\left(1-\rho w\right)-\ds\f{wv}{(1+\alpha \rho w)(1+\beta v)}, & x\in \Omega,\\
-d_2\Delta v=-dv+\ds\f{wv}{(1+\alpha \rho w )(1+\beta v)}, & x\in\Omega,\\
 \partial_\nu  w=\partial_\nu
  v=0,& x\in \partial \Omega,\\
\end{cases}
\end{equation}
and $(u,z)$ satisfies
\begin{equation}\label{uz}
\begin{cases}
-d_1\Delta u=u\left(1-u\right)-\ds\f{uz}{(1+u)(\rho+\beta z)}, & x\in \Omega,\\
-d_2\Delta z=-dz+\ds\f{uz}{(1+u)(\rho+\beta z)}, & x\in\Omega,\\
 \partial_\nu  u=\partial_\nu
  z=0,& x\in \partial \Omega.\\
\end{cases}
\end{equation}
Therefore, the existence/nonexistence of positive solutions of system \eqref{stCM} for large $c$ is equivalent to that of system \eqref{wv} or \eqref{uz} for small $\rho$.
The method used here is motivated by \cite{Peng-Shi}.
For later applications, we cite the following three well-known results. The first is from  \cite{Lieberman,Peng-Shi}.

\begin{lemma}\label{cc1}
Assume that $\Omega$ is a bounded Lipschitz domain in $\mathbb{R}^N$, $d$ is a nonnegative constant, and
$z \in W^{1,2}(\Omega)$ is a non-negative weak solution of the following inequalities
\begin{equation*}
\begin{cases}
  -\Delta z +d z\ge0, & x\in \Omega,\\
 \partial_\nu  z\le0,& x\in \partial \Omega.\\
\end{cases}
\end{equation*}
Then, there is a positive constant $C$, which is determined only by $d$ and $\Omega$, such that
$$\int_\Omega z dx \le C\inf_{x\in\Omega} z.$$
\end{lemma}

Then, we cite a Harnack inequality from \cite{Lin-Ni,Peng-Shi-Wang2}.
\begin{lemma}\label{harnack}
Assume that $\Omega$ is a bounded Lipschitz domain in $\mathbb{R}^N$, $c(x)\in L^q(\Omega)$ for some $q > N/2$, and
$z \in W^{1,2}(\Omega)$ is a non-negative weak solution of the following problem
\begin{equation*}
\begin{cases}
  \Delta z +c(x) z=0, & x\in \Omega,\\
 \partial_\nu  z=0,& x\in \partial \Omega.\\
\end{cases}
\end{equation*}
Then, there is a positive constant $C$, which is determined only by $\|c(x)\|_q$, $q$, and $\Omega$, such that
$$\sup_{x\in\Omega} z\le C\inf_{x\in\Omega} z.$$
\end{lemma}
Finally, we cite a maximum principle from \cite{Lou-Ni}.
\begin{lemma}\label{maxin}
Assume that $\Omega$ is a bounded smooth domain in $\mathbb{R}^N$, $g\in C(\overline\Omega\times \mathbb R)$, and $z\in C^{2}(\Omega)\cap C^1(\overline\Omega)$
satisfies the following inequalities
\begin{equation*}
\begin{cases}
  \Delta z+g(x,z)\ge0, & x\in \Omega,\\
 \partial_\nu  z\le0,& x\in \partial \Omega.\\
\end{cases}
\end{equation*}
If $z(x_0)=\max_{x\in\overline \Omega}z$, then $g(x_0,z(x_0))\ge0$.
\end{lemma}
It follows from Theorem \ref{equ} that if
\begin{equation}\label{ca2}
\alpha>1\;\;\text{and}\;\; 0<\beta<\ds\f{4\alpha}{(\alpha+1)^2},\;\;\text{(case I)}\end{equation} then system \eqref{stCM} has only one constant positive solution for sufficiently large $c$, and if
\begin{equation}\label{ca1}
\alpha>1\;\;\text{and}\;\;\ds\f{4\alpha}{(\alpha+1)^2}<\beta<1,\;\;\text{(case II)}
\end{equation}
 then system \eqref{stCM} has
three constant positive solutions for sufficiently large $c$. Therefore, the following discussion is divided into two cases.
By using Lemmas \ref{cc1}-\ref{maxin}, we first give two results on a \textit{priori} estimates for positive solutions of system \eqref{stCM}.
\begin{lemma}\label{asy}
Let $(u_{i}(x),v_{i}(x))$ be a positive solution of system \eqref{stCM} for $c=c_i$, where $i=1,2,\cdots$, and $\lim_{i\to\infty}c_i=\infty$. Assume that one of the following assumptions is satisfied:
\begin{enumerate}
\item [(i)]
$\alpha$ and $\beta$ satisfy Eq. \eqref{ca2}.
\item [(ii)] $\alpha$ and $\beta$ satisfy Eq. \eqref{ca1}, and $u_{i}(x)\to0$ in $C(\overline \Omega)$ as $i\to \infty$.
\end{enumerate}
Then, there exists a subsequence $\{i_k\}_{k=1}^\infty$ such that
$(c_{i_k}u_{i_k}(x),v_{i_k}(x))\to (\tilde w(x),\tilde v(x))$ in $C^2(\overline \Omega)$ as $k\to \infty$, where $(\tilde w(x), \tilde v(x))$ is a positive solution of system \eqref{wv} for $\rho=0$.
\end{lemma}
\begin{proof}
First, we derive the existence of the upper bounds for $\{c_iu_i\}$ and $\{v_i\}$.
Let $w_i=c_iu_i$ and $\rho_i=1/c_i$. Then $(w_i,v_i)$ satisfies
\begin{equation}\label{stCMu}
\begin{cases}
-d_1\Delta w_i=w_i\left(1-\rho_i w_i\right)-\ds\f{w_iv_i}{(1+\alpha \rho_i w_i)(1+\beta v_i)}, & x\in \Omega,\\
-d_2\Delta v_i=-dv_i+\ds\f{w_iv_i}{(1+\alpha \rho_i w_i )(1+\beta v_i)}, & x\in\Omega,\\
 \partial_\nu  w_i=\partial_\nu
  v_i=0,& x\in \partial \Omega.\\
\end{cases}
\end{equation}
Thanks to Lemma \ref{cc1}, there is a positive constant $C_0$ such that
\begin{equation}\label{c0}
\int_\Omega v_{i} dx \le C_0\inf_{x\in\Omega} v_i\;\;\text{for all}\;\; i\ge1.
\end{equation}
We claim that there exists $C_1>0$ such that
\begin{equation}\label{c1}
 \int_\Omega v_{i} dx \le C_1\;\;\text{for all}\;\; i\ge1.
\end{equation}
Suppose that it is not true. Then there exists a subsequence $\{i_n\}_{n=1}^\infty$ such that $\lim_{n\to\infty}i_n=\infty$ and
$\lim_{n\to\infty}\int_\Omega v_{i_n} dx=\infty$, which implies that $v_{i_n}\to\infty$ uniformly on $\overline \Omega$ as $n\to\infty$ from Eq. \eqref{c0}.
By virtue of Lemma \ref{maxin}, we have
\begin{equation*}
\sup_{x\in\Omega}u_{i}\le1\;\;\text{for all}\;\;i\ge1.
\end{equation*}
If assumption $(i)$ is satisfied, then, for sufficiently large $n$,
\begin{equation*}
\begin{split}
-d_1\Delta w_{i_n}=& \ds\f{w_{i_n}}{1+\alpha u_{i_n}}\left[(1- u_{i_n})(1+\alpha u_{i_n})-\ds\f{v_{i_n}}{1+\beta v_{i_n}}\right]\\
\le&\ds\f{1}{2}\left[\ds\f{(\alpha+1)^2}{4\alpha}-\ds\f{1}{\beta}\right]\ds\f{w_{i_n}}{1+\alpha},
\end{split}
\end{equation*}
which implies that $w_{i_n}\le0$ for sufficiently large $n$. If assumption $(ii)$ is satisfied, then, for sufficiently large $n$,
\begin{equation*}
\begin{split}
-d_1\Delta w_{i_n}=& \ds\f{w_{i_n}}{1+\alpha u_{i_n}}\left[(1- u_{i_n})(1+\alpha u_{i_n})-\ds\f{v_{i_n}}{1+\beta v_{i_n}}\right]\\
\le&\ds\f{1}{2}\left(1-\ds\f{1}{\beta}\right)w_{i_n},
\end{split}
\end{equation*}
which also leads to $w_{i_n}\le0$ for sufficiently large $n$. Therefore, the contradiction is arrived for both cases, and Eq. \eqref{c1} holds. By the second equation of \eqref{stCMu}, we have
\begin{equation}\label{inw}
\ds\f{|\Omega|}{1+\alpha}\inf_{x\in\Omega} w_i\le\int_{\Omega}\ds\f{ w_i}{1+\alpha u_i}dx\le d\int_{\Omega}(1+\beta v_i)dx\;\;\text{for all}\;\; i\ge1.
\end{equation}
Because $$\left\|1-u_i-\ds\f{v_i}{(1+\alpha u_i)(1+\beta v_i)}\right\|_\infty\le 2+\ds\f{1}{\beta}\;\;\text{for all}\;\;i\ge1,$$
by Lemma \ref{harnack}, we see that there exists a positive constant $C_2$ such that
\begin{equation}\label{c2}
\sup_{x\in\Omega}w_{i}\le C_2\inf_{x\in\Omega} w_{i}\;\;\text{for all}\;\;i\ge1.
\end{equation}
It follows from Eqs. \eqref{c1}-\eqref{c2} that there exists a positive constant $C_3$ such that
\begin{equation}\label{c3}
\sup_{x\in\Omega}w_{i}\le C_3\;\;\text{for all}\;\;i\ge1,
\end{equation}
which yields
$$\left\|-d+\ds\f{w_i}{(1+\alpha u_i)(1+\beta v_i)}\right\|_\infty\le d+C_3\;\;\text{for all}\;\;i\ge1.$$
Again, from Lemma \ref{harnack}, we obtain that there exists a positive constant $C_4$ such that
\begin{equation}\label{c4}
\sup_{x\in\Omega}v_{i}\le C_4\inf_{x\in\Omega} v_{i}\;\;\text{for all}\;\;i\ge1.
\end{equation}
By virtue of Eqs. \eqref{c1} and \eqref{c4}, we see that there exists a positive constant $C_5$ such that
\begin{equation}\label{c5}
\sup_{x\in\Omega}v_{i}\le C_5\;\;\text{for all}\;\;i\ge1.
\end{equation}

Then, we find the lower bounds for $\{w_{i}\}$ and $\{v_{i}\}$. We first claim that there exists a positive constant $C_6$ such that
\begin{equation}\label{c6}
\inf_{x\in\Omega} w_i \ge C_6\;\;\text{for all}\;\;i\ge1.
\end{equation}
Suppose Eq. \eqref{c6} does not holds. Then there exists a subsequence $\{i_m\}_{m=1}^\infty$ such that $\lim_{m\to\infty}i_m=\infty$ and
$\lim_{m\to\infty} \inf_{x\in\Omega} w_{i_m}=0$. By Eq. \eqref{c2}, we have $w_{i_{m}}\to0$ uniformly on $\overline \Omega$ as $ m\to \infty$. Hence, for sufficiently large $m$,
$$\int_{\Omega}v_{i_m}\left[d-\ds\f{w_{i_m}}{(1+\alpha u_{i_m})(1+\beta v_{i_m})}\right]dx>0,$$
which is a contradiction. Hence Eq. \eqref{c6} holds. Then, we claim that there exists a positive constant $C_7$ such that
\begin{equation}\label{c7}
\inf_{x\in\Omega} v_i \ge C_7\;\;\text{for all}\;\;i\ge1.
\end{equation}
Suppose that Eq. \eqref{c7} does not hold. Then, there exists a subsequence $\{i_j\}_{j=1}^\infty$ such that $\lim_{j\to\infty}i_j=\infty$ and
$\lim_{j\to\infty} \inf_{x\in\Omega} v_{i_j}=0$, which leads to $v_{i_{j}}\to0$ uniformly on $\overline \Omega$ as $ j\to \infty$ from Eq. \eqref{c4}. Noticing that $\{w_i\}$ is bounded,  we have, for sufficiently large $j$,
$$\int_{\Omega}w_{i_j}\left[1-\rho_{i_j}w_{i_j}+\ds\f{v_{i_j}}{(1+\alpha u_{i_j})(1+\beta v_{i_j})}\right]dx>0,$$ which is a contradiction. Therefore, Eq. \eqref{c7} holds.

Finally, we give the asymptotic behavior of $\{w_i\}$ and $\{v_i\}$. From above analysis, we see that both $\{w_i\}$ and $\{v_i\}$ are bounded. Then, due to the $L^p$ theory, we obtain that $\{w_i\}$ and $\{v_i\}$ are bounded in $W^{2,p}(\Omega)$ for any $p>N$. It follows from the embedding theorem that $\{w_i\}$ and $\{v_i\}$ are precompact in $C^1(\overline\Omega)$.
Then, there exists a subsequence $\{i_k\}_{k=1}^{\infty}$  and $(\tilde w(x),\tilde v(x))\in C^1(\overline \Omega)\times C^1(\overline \Omega)$ such that
$$(cu_{i_k},v_{i_k})=(w_{i_k},v_{i_k})\to (\tilde w(x),\tilde v(x))\;\;\text{in}\;\;C^1(\overline \Omega)\times C^1(\overline \Omega)\text{ as }k\to\infty,$$
where $\tilde w(x)$ and $\tilde z(x)$ are positive from Eqs. \eqref{c6} and \eqref{c7}.
Note that
\begin{equation}\label{abv}
\begin{split}
w_{i_k}= &[-d_1\Delta+I]^{-1}\left[ w_{i_k}+w_{i_k}\left(1-\rho_{i_k} w_{i_k}-\ds\f{z_{i_k}}{(1+\rho_{i_k}w_{i_k})(1+\beta v_{i_k})}\right)\right],\\
v_{i_k}=&[-d_2\Delta+I]^{-1}\left[ v_{i_k}+v_{i_k}\left(-d+ \ds\f{w_{i_k}}{(1+\rho_{i_k}w_{i_k})(1+\beta v_{i_k})}\right)\right],
\end{split}
\end{equation}
and $\lim_{k\to\infty}\rho_{i_k}w_{i_k}=0$ in $C^1(\overline \Omega)$.
Then, taking the limit of Eq. \eqref{abv}
as $k\to\infty$ and by the Schauder theorem, we see that
$(\tilde w(x),\tilde v(x))$ is a positive solution of system \eqref{wv} for $\rho=0$, and
$$(cu_{i_k},v_{i_k})=(w_{i_k},v_{i_k})\to (\tilde w(x),\tilde v(x))\;\;\text{in}\;\;C^2(\overline \Omega)\times C^2(\overline \Omega)\text{ as }k\to\infty,$$
The proof is complete.
\end{proof}

\begin{lemma}\label{asy2}
Let $(u_{i}(x),v_{i}(x))$ be a positive solution of system \eqref{stCM} for $c=c_i$, where $i=1,2,\cdots$, and $\lim_{i\to\infty}c_i=\infty$. Assume that $\alpha$ and $\beta$ satisfy Eq. \eqref{ca1}, and $u_{i}(x)\to \tilde u(x)$ in $C(\overline \Omega)$ as $i\to \infty$, where $\tilde u(x)>0$ for $x\in\overline \Omega$.
Then there exists a subsequence $\{i_k\}_{k=1}^\infty$ such that
$(u_{i_k}(x),v_{i_k}(x)/c_{i_k})\to (\tilde u(x),\tilde z(x))$ in $C^2(\overline \Omega)$ as $k\to \infty$, where $(\tilde u(x), \tilde z(x))$ is a positive solution of system \eqref{uz} for $\rho=0$.
\end{lemma}
\begin{proof}
Let $z_i=v_i/c_i$ and $\rho_i=1/c_i$, and then $(u_i,z_i)$ satisfies
\begin{equation}\label{stCM22}
\begin{cases}
-d_1\Delta u_i=u_i\left(1-u_i\right)-\ds\f{u_iz_i}{(1+u_i)(\rho_i+\beta z_i)}, & x\in \Omega,\\
-d_2\Delta z_i=-dz_i+\ds\f{u_iz_i}{(1+u_i)(\rho_i+\beta z_i)}, & x\in\Omega,\\
 \partial_\nu  u_i=\partial_\nu
  z_i=0,& x\in \partial \Omega.\\
\end{cases}
\end{equation}
It is deduced by Lemma \ref{maxin} that
\begin{equation}\label{u1}
\sup_{x\in\Omega}u_{i}\le1\;\;\text{for all}\;\;i\ge1,
\end{equation}
which yields
\begin{equation}\label{z1}
\sup_{x\in\Omega}z_{i}\le\ds\f{1}{d\beta}\;\;\text{for all}\;\;i\ge1.
\end{equation}
Since $u_{i}(x)\to \tilde u(x)$ in $C(\overline \Omega)$ as $i\to \infty$, we see that there exists a positive constant $C_1$ such that
\begin{equation}\label{u2}
\inf_{x\in\Omega}u_{i}\ge C_1\;\;\text{for all}\;\;i\ge1.
\end{equation}
Consequently, by the second equation of \eqref{stCM22}, we derive a positive constant $C_2$ satisfying
\begin{equation}\label{z2}
\inf_{x\in\Omega}z_{i}\ge C_2\;\;\text{for all}\;\;i\ge1.
\end{equation}

Finally, we give the limit profile of $\{u_i\}$ and $\{z_i\}$. Similarly to the arguments in the proof of Lemma \ref{asy},
we see that there exists a subsequence $\{i_k\}_{k=1}^{\infty}$  and $(\tilde u(x),\tilde z(x))\in C^1(\overline \Omega)\times C^1(\overline \Omega)$ such that
$$(u_{i_k},v_{i_k}/c_{i_k})=(u_{i_k},z_{i_k})\to (\tilde u(x),\tilde z(x))\;\;\text{in}\;\;C^1(\overline \Omega)\times C^1(\overline \Omega)\text{ as }k\to\infty,$$
where $\tilde w(x)$ and $\tilde z(x)$ are positive from Eqs. \eqref{u2} and \eqref{z2}.
Taking the limit of the following equation
\begin{equation*}
\begin{split}
u_{i_k}= &[-d_1\Delta+I]^{-1}\left[ u_{i_k}+u_{i_k}\left(1- u_{i_k}-\ds\f{z_{i_k}}{(1+u_{i_k})(\rho_{i_k}+\beta z_{i_k})}\right)\right],\\
z_{i_k}=&[-d_2\Delta+I]^{-1}\left[ z_{i_k}+z_{i_k}\left(-d+ \ds\f{u_{i_k}}{(1+u_{i_k})(\rho_{i_k}+\beta z_{i_k})}\right)\right],
\end{split}
\end{equation*}
as $k\to\infty$, we see that
$(\tilde u(x),\tilde z(x))$ is a positive solution of system \eqref{uz} for $\rho=0$.
This completes the proof.
\end{proof}
Now, based on the above two lemmas, we establish the results concerning with the nonexistence of nonconstant steady states for large $c$.
We first consider the case that $\alpha$ and $\beta $ satisfy Eq. \eqref{ca2} (case I).
\begin{theorem}\label{nn1}
Assume that $$\alpha>1\;\;\text{and}\;\; 0<\beta<\ds\f{4\alpha}{(\alpha+1)^2}.$$ Then there exists a positive constant $c_*=c_*(d_1,d_2,\alpha,\beta,d,\Omega)$ such that, for $c>c_*$, system \eqref{CM} has a unique constant positive steady state and no
nonconstant positive steady states.
\end{theorem}
\begin{proof}
We argue indirectly and assume that there exists $\{c_i\}_{i=1}^\infty$ such that $\lim_{i\to\infty} c_i=\infty$, and system \eqref{CM} has a nonconstant positive steady state $(u_{i}(x),v_{i}(x))$ for any $c=c_i$. Then, owing to Lemma \ref{asy}, there exists a subsequence $\{i_k\}_{k=1}^\infty$ such that
$(c_{i_k}u_{i_k}(x),v_{i_k}(x))\to (\tilde w(x),\tilde v(x))$ in $C^2(\overline \Omega)$ as $k\to \infty$, where $(\tilde w(x), \tilde v(x))$ is a positive solution of system \eqref{wv} for $\rho=0$.

For $\rho=0$, system \eqref{wv} has a unique constant positive steady state $$(\hat w,\hat v)=\left(\ds\f{d}{1-\beta},\ds\f{1}{1-\beta}\right).$$
Set
\begin{equation*}
\begin{split}
G(w,v):=&\int_{\Omega}\left\{\ds\f{w-\hat w}{w}\left[d_1\Delta w+w\left(1-\ds\f{ v}{1+\beta v}\right)\right]\right \}dx\\
+&\int_{\Omega}\left\{\ds\f{\phi_2(v)-\phi_2(\hat v)}{\phi_2(v)}\left[d_2\Delta v+\phi_2(v)\left(-d-d\beta v+w\right)\right]\right\}dx,\\
\end{split}
\end{equation*}
where $\phi_2(v)$ is defined as in Eq. \eqref{pspi}. Through a direct calculation, we see that if $(w,v)$ satisfies system \eqref{wv} for $\rho=0$, then
\begin{equation}\label{G}
\begin{split}
G(w,v)=&-d_1\hat w\int_{\Omega}\ds\f{|\nabla w|^2}{w^2}dx-d_2\phi_2(\hat v)\int_\Omega\ds\f{\phi'_2(v)}{[\phi_2(v)]^2}|\nabla v|^2dx\\-&d\beta \int_\Omega \ds\f{(v-\hat v)^2}{(1+\beta v)(1+\beta \hat v)}dx.
\end{split}
\end{equation}
Then, $(\tilde w(x),\tilde v(x))\equiv\left(\hat w,\hat v\right)$, and consequently, $$(c_{i_k}u_{i_k}(x),v_{i_k}(x))\to \left(\hat w,\hat v\right)\;\;\text{in}\;\;C^2(\overline \Omega),$$as $k\to \infty$.

By the careful calculation, we can see that all the eigenvalues of $\left(\hat w,\hat v\right)$ are negative for the corresponding parabolic equation of system \eqref{wv} when $\rho=0$. Then, taking advantage of the implicit theorem, there exists $\rho_0>0$ such that, for $\rho<\rho_0$, system \eqref{wv} has a unique solution in the neighborhood of $\left(\hat w,\hat v\right)$ in $C^1(\overline \Omega)$, and this solution is constant and locally asymptotically stable for the corresponding parabolic equation. It follows that $(u_{i_k}(x),v_{i_k}(x))$ is constant for sufficiently large $k$, which is a contradiction. The proof is complete.
\end{proof}
Then, we consider the case that $\alpha$ and $\beta $ satisfy Eq. \eqref{ca1} (case II).
\begin{theorem}\label{nn2}
Assume that $$\alpha>1,\;\ds\f{4\alpha}{(\alpha+1)^2}<\beta<1\text{ and } d_1>1/\mu_1,$$
where $\mu_1$ is defined as in Eq. \eqref{mmu}.
Then there exists a positive constant $c_*=c_*(d_1,d_2,\alpha,\beta,d,\Omega)$ such that, for $c>c_*$, system \eqref{CM} has three constant positive steady states and no nonconstant positive steady states.
\end{theorem}
\begin{proof}
Suppose on the contrary that there exists $\{c_i\}_{i=1}^\infty$ such that $\lim_{i\to\infty} c_i=\infty$, and system \eqref{CM} has a nonconstant positive steady state $(u_{i}(x),v_{i}(x))$ for any $c=c_i$. By the arguments similar to \cite{Du-HSU2}, we first show that there exists a subsequence $\{i_k\}_{k=1}^\infty$ such that
\begin{equation*}
\begin{split}
&\text{ Case 1: } u_{i_k}(x)\to 0 \text{ in }C^1(\overline \Omega) \text{ or}\\
&\text{ Case 2: } u_{i_k}(x)\to \tilde u(x) \text{ in }C^1(\overline \Omega), \text{ where }\tilde u(x)>0\text{ for }x\in\overline \Omega,
\end{split}
\end{equation*}
as $k\to\infty$.
Denote $$f_i(x)=1-u_i(x)-\ds\f{v_i(x)}{(1+\alpha u_i(x))(1+\beta v_i(x))}.$$
Since $\sup_{x\in\Omega}u_{i}\le1$ for all $i\ge1$, we have $\|f_i(x)\|_\infty\le 2+1/\beta$ for all $i\ge1$.
It follows from the $L^p$ theory that $\{u_i\}$ is bounded in $ W^{2,p}(\Omega)$ for any $p > N$. Consequently, by the embedding theorem,  $\{u_i\}$ is precompact in $C^1(\overline \Omega)$. Then, there exists a subsequence $\{i_k\}_{k=1}^\infty$ such that
$u_{i_k}(x)\to \tilde u(x)$ in $C^1(\overline \Omega)$ and $f_{i_k}(x)\to f(x) $ weakly in $L^2(\Omega)$ as $k\to \infty$. We note that $\| f(x)\|_\infty\le 2+1/\beta$ since each $f_{i}$ has this property. Therefore, $\tilde u$ is a weak solution
of the following equation
\begin{equation}\label{G22}
\begin{cases}
-d_1\Delta  u=f(x)u, & x\in \Omega,\\
 \partial_\nu  u=0,& x\in \partial \Omega.\\
\end{cases}
\end{equation}
Noticing that $f(x)\in L^\infty(\Omega)$, we have $\tilde u(x)\equiv 0$ or $\tilde u(x)>0$ for $x\in\overline \Omega$. Then the following discussion is divided into two case.

Case 1: $u_{i_k}(x)\to 0 \text{ in }C^1(\overline \Omega)$ as $k\to\infty$. We denote $u_{i_k}$ by $u_i$ for convenience. It follows from Lemma \ref{asy} that there exists a subsequence $\{i_n\}_{n=1}^\infty$ such that
$$(c_{i_n}u_{i_n}(x),v_{i_n}(x))\to (\tilde w(x),\tilde v(x))$$ in $C^2(\overline \Omega)$ as $n\to \infty$, where $(\tilde w(x), \tilde v(x))$ is a positive solution of system \eqref{wv} for $\rho=0$. By the arguments similar to Theorem \ref{nn1}, we see that $(\tilde w(x),\tilde v(x))\equiv \left(\ds\f{d}{1-\beta},\ds\f{1}{1-\beta}\right)$, and $(c_{i_n}u_{i_n}(x),v_{i_n}(x))$ is constant for sufficiently large $n$, which is a contradiction.

Case 2: $u_{i_k}(x)\to \tilde u(x)$ in $C^1(\overline \Omega)$ as $k\to\infty$, where $\tilde u(x)>0$ for $x\in\overline \Omega$.
We also denote $u_{i_k}$ by $u_i$ for convenience. Due to Lemma \ref{asy2}, there exists a subsequence $\{i_n\}_{n=1}^\infty$ such that
$(u_{i_n}(x),v_{i_n}(x)/c_{i_n})\to (\tilde u(x),\tilde z(x))$ in $C^2(\overline \Omega)$ as $n\to \infty$, where $(\tilde u(x),\tilde z(x))$ is a positive solution of \eqref{uz} for $\rho=0$.
Then, we consider the steady states of system \eqref{uz} for $\rho=0$, which satisfy
\begin{equation}\label{uzred}
\begin{cases}
-d_1\Delta u=u\left(1-u\right)-\ds\f{u}{\beta(1+\alpha u)}, & x\in \Omega,\\
-d_2\Delta z=-dz+\ds\f{u}{\beta(1+\alpha u)}, & x\in\Omega,\\
 \partial_\nu  u=\partial_\nu
  z=0,& x\in \partial \Omega.\\
\end{cases}
\end{equation}
Clearly, Eq. \eqref{uzred} has two constant positive steady states, denoted by $(\hat u_1,\hat z_1)$ and $(\hat u_2,\hat z_2)$.
Denote $\overline u=\ds\f{1}{|\Omega|}\int_\Omega udx$. Then, multiplying the first equation of \eqref{uzred} by $u-\overline u$, and integrating the result over $\Omega$, we have \begin{equation*}
\begin{split}
&d_1\int_\Omega |\nabla (u-\overline u)|^2dx\\
=&\int_\Omega (u-\overline u)\left(u\left(1-u\right)-\ds\f{u}{\beta(1+\alpha u)}-\overline u\left(1-\overline u\right)+\ds\f{\overline u}{\beta(1+\alpha \overline u)}\right)dx\\
\le& \int_\Omega (u-\overline u)^2dx.
\end{split}
\end{equation*}
This, combined with the Poincar\'e inequality, yields
$$d_1\mu_1\int_\Omega (u-\overline u)^2dx\le d_1\int_\Omega |\nabla (u-\overline u)|^2dx\le \int_\Omega (u-\overline u)^2dx.$$
Noticing that $d_1>1/\mu$, we have $u(x)\equiv \overline u$, and hence $\tilde u(x)\equiv \hat u_1$ or $\tilde u(x)\equiv \hat u_2$, which implies that
$(u_{i_n}(x),v_{i_n}(x)/c_{i_n})\to ( \hat u_1,\hat z_1) \text{ or } ( \hat u_2,\hat z_2)$
in $C^2(\overline \Omega)$ as $n\to \infty$. By the careful calculation, we obtain that zero is not the eigenvalue of the linearized problem for Eq. \eqref {uzred} with respect to $(\hat u_i,\hat z_i)$ for $i=1,2$. By the implicit theorem, we see that, for each $i=1,2$, there exists $\rho_i>0$ such that system \eqref{uz} has a unique positive solution in the neighborhood of $\left(\hat u_i,\hat z_i\right)$ in $C^1(\overline \Omega)$ for $\rho<\rho_i$. Therefore, $(u_{i_k}(x),v_{i_k}(x))$ is constant for sufficiently large $k$, which is a contradiction.
\end{proof}
At the end of this section, we show the nonexistence of nonconstant positive steady states when diffusion rates $d_1$ and $d_2$ are large. This result will be used in the next section, and the arguments are similar to \cite{NiW,WangSW}.
\begin{theorem}\label{coef}
There exists a positive constant $d_*=d_*(\alpha,\beta,c,d,\Omega)$ such that system \eqref{stCM} has no nonconstant positive solutions for $d_1,d_2\ge d_*$.
\end{theorem}
\begin{proof}
Let $(u,v)$ be a positive solution of system \eqref{stCM}, and denote
$$\overline u=\ds\f{1}{|\Omega|}\int_\Omega udx,\;\;\overline v=\ds\f{1}{|\Omega|}\int_\Omega vdx.$$
By Lemma \ref{maxin}, we have $0<u\le1$ for $x\in\overline\Omega$, which leads to $0<\overline u\le1$.
Noticing that $ c\int_\Omega u(1-u)dx=d\int_\Omega v dx$, we have $\overline v\le\ds\f{c}{d}$.
Then, multiplying the first equation of system \eqref{stCM} by $u-\overline u$, and integrating the result over $\Omega$, we have
\begin{equation*}
\begin{split}
&d_1\int_\Omega|\nabla (u-\overline u)|^2dx\\
=&\int_\Omega (u-\overline u)[u(1-u)-\overline u(1-\overline u)]dx\\
-&\int_\Omega (u-\overline u)\left[\ds\f{uv}{(1+\alpha u)(1+\beta v)}-\ds\f{\overline u \;\overline v}{(1+\alpha \overline u)(1+\beta \overline v)}\right]dx\\
\le &\int_\Omega (u-\overline u)^2 dx-\int_\Omega \ds\f{\overline u(u-\overline u)(v-\overline v)}{(1+\alpha u)(1+\beta v)(1+\beta \overline v)}dx
+\int_\Omega \ds\f{\alpha \overline u\;\overline v(u-\overline u)^2}{(1+\alpha\overline u )(1+\alpha u)(1+\beta \overline v)}dx\\
\le&\left( \ds\f{3}{2}+\ds\f{c\alpha}{d}\right)\int_\Omega (u-\overline u)^2 dx+\ds\f{1}{2}\int_\Omega (v-\overline v)^2 dx.
\end{split}
\end{equation*}
Similarly, multiplying the second equation of system \eqref{stCM} by $v-\overline v$, and integrating the result over $\Omega$, we get
\begin{equation*}
\begin{split}
&d_2\int_\Omega|\nabla (v-\overline v)|^2dx\\
=&\int_\Omega (v-\overline v)\left[-dv+d\overline v+\ds\f{cuv}{(1+\alpha u)(1+\beta v)}-\ds\f{c\overline u \;\overline v}{(1+\alpha \overline u)(1+\beta \overline v)}\right]dx\\
\le&\ds\f{c}{1+\alpha}\int_\Omega (v-\overline v)^2dx+\int_\Omega \ds\f{c\overline v(v-\overline v)(u-\overline u)}{(1+\alpha u)(1+\beta v)(1+\alpha \overline u)}dx\\
\le&\left(\ds\f{c}{1+\alpha}+\ds\f{c^2}{2d}\right)\int_\Omega (v-\overline v)^2dx+\ds\f{c^2}{2d}\int_\Omega (u-\overline u)^2dx.
\end{split}
\end{equation*}
Denote
\begin{equation*}
A=\ds\f{3}{2}+\ds\f{c\alpha}{d}+\ds\f{c^2}{2d},\;\;\text{and}\;\;
B=\ds\f{c}{1+\alpha}+\ds\f{c^2}{2d}+\ds\f{1}{2}.
\end{equation*}
Then, due to the Poincar\'e inequality, we have
\begin{equation}
\begin{split}
&d_1\int_\Omega|\nabla (u-\overline u)|^2dx+d_2\int_\Omega|\nabla (v-\overline v)|^2dx\\
\le &\ds\f{A}{\mu_1}\int_\Omega|\nabla (u-\overline u)|^2dx+\ds\f{B}{\mu_1}\int_\Omega|\nabla (v-\overline v)|^2dx.
\end{split}
\end{equation}
Therefore, if $\min\{d_1,d_2\}>\ds\f{1}{\mu_1}\max\{A,B\}$, then
$$\nabla (u-\overline u)=\nabla (v-\overline v)\equiv 0,$$ which implies that
$u$ and $v$ are both constants.
\end{proof}
\subsection{The existence}
In this subsection, we shall use the Leray-Schauder degree theory to investigate the existence of nonconstant positive solutions of system \eqref{stCM}. Recall that we assume $\alpha>1$ and $\beta<1$ throughout the whole section.
The arguments here are motivated by \cite{NiW,pangw}. First we derive a \textit{priori} upper and lower bounds for
positive solutions of system \eqref{stCM}.
\begin{lemma}\label{les}
Assume that $c>d(1+\alpha)$. Let $\underline {d}_1\le d_1\le \overline{d}_1$ and $d_2\ge\underline{d}_2$, where $\underline{d}_1$, $\overline{d}_1$ and $\underline{d}_2$ are positive constants, and $(u(x),v(x))$ be a positive solution of system \eqref{stCM}. Then, there exist two positive constants
$\underline C=\underline C(\underline{d}_1,\overline{d}_1,\underline{d}_2,d,\alpha,\beta,c)$ and $\overline C=\overline C(\beta,c,d)$ such that
$$\underline C\le \inf_{x\in\Omega}u(x)\le\sup_{x\in\Omega}u(x)\le\overline C,\;\text{ and }\;\underline C\le\inf_{x\in\Omega}v(x)\le\sup_{x\in\Omega}v(x)\le\overline C,$$
for all $d_1\in[\underline{d}_1,\overline{d}_1]$ and $d_2\ge\underline{d}_2$.
\end{lemma}
\begin{proof}
By Lemma \ref{maxin}, we have $0<u\le1$ for $x\in\overline\Omega$. Consequently,
$-d_1\Delta u\le -dv+c/\beta$, and hence $\sup_{x\in\Omega}v(x)\le c/d\beta$ from Lemma \ref{maxin}. Let $\overline C(\beta,c,d)=\max\{1,c/d\beta\}$. Then we have
\begin{equation}\label{upb}
\sup_{x\in\Omega}u(x),\sup_{x\in\Omega}v(x)\le\overline C,
\end{equation}
which leads to
\begin{equation*}
\begin{split}
&\ds\f{1}{d_1}\left\|1-u-\ds\f{v}{(1+\alpha u)(1+\beta v)}\right\|_\infty\le\ds\f{2+\overline C}{\underline{d}_1}\;\text{for all}\;\;d_1\in[\underline{d}_1,\overline{d}_1],\\
&\ds\f{1}{d_2}\left\|-d+\ds\f{u}{(1+\alpha u)(1+\beta v)}\right\|_\infty\le\ds\f{d+\overline C}{\underline{d}_2}\;\text{for all}\;\;d_2\ge\underline{d}_2.
\end{split}
\end{equation*}
Then it follows from Lemma \ref{harnack} that there exists a positive constant $C_1$ such that
\begin{equation}\label{ha}
\sup_{x\in\Omega}u(x)\le C_1\inf_{x\in\Omega}u(x) \;\;\text{and}\;\;\sup_{x\in\Omega}v(x)\le C_1\inf_{x\in\Omega}v(x),
\end{equation}
for all $d_1\in[\underline{d}_1,\overline{d}_1]$ and $d_2\ge \underline{d}_2$.
Now, we derive the lower bounds for $u$ and $v$. In fact, we claim that there exists  a positive constant $C(\underline{d}_1,\overline{d}_1,\underline{d}_2,d,\alpha,\beta,c)$ such that
$$\inf_{x\in\Omega}u(x),\inf_{x\in\Omega}v(x)\ge \underline C,$$
  for all $d_1\in[\underline{d}_1,\overline{d}_1]$ and $d_2\ge \underline{d}_2$.
 If it is not true, then there exists a sequence $\{(d_1^{(i)},d_2^{(i)})\}_{i=1}^\infty$ such that the corresponding solution $(u_i(x),v_i(x))$ for $d_1=d_1^{(i)}$ and $d_2=d_2^{(i)}$ satisfies
$$\lim_{i\to\infty}\inf_{x\in\Omega}u_i(x)=0 \;\;\text{or}\;\;\lim_{i\to\infty}\inf_{x\in\Omega}v_i(x)=0.$$
We first consider the case that $\lim_{i\to\infty}\inf_{x\in\Omega}u_i(x)=0$. By virtue of Eq. \eqref{ha}, we have
$\lim_{i\to\infty}u_i(x)=0$ in $C(\overline\Omega)$, which implies that
$$\int_\Omega v_i\left[d-\ds\f{cu_i}{(1+\alpha u_i)(1+\beta v_i)}\right]dx>0\;\; \text{for sufficiently large }i. $$
This is a contradiction. Then we consider the case that $\lim_{i\to\infty}\inf_{x\in\Omega}v_i(x)=0$. Again, by Eq. \eqref{ha}, we have $\lim_{i\to\infty}v_i(x)=0$ in $C(\overline\Omega)$. Due to the $L^p$ theory and
embedding theorem, there exists a subsequence $\{i_k\}_{k=1}^\infty$ such that
$\lim_{k\to\infty}d_{i_k}=d_0$ and $\lim_{k\to\infty}u_{i_k}(x)=1$ in $C^1(\overline\Omega)$.
From the second Equation of \eqref{stCM}, we get
$$\int_{\Omega}\ds\f{cu_{i_k}}{1+\alpha u_{i_k}}dx\le\int_\Omega d(1+v_{i_k})dx.$$
Taking the limit of the above equation as $k\to\infty$, we have $\ds\f{c}{1+\alpha}\le d$, which contradicts with $c>d(1+\alpha)$.
This completes the proof.
\end{proof}
As in \cite{NiW}, Define $$\mathbf{X}=\{\mathbf u=(u,v)\in C^1(\overline\Omega)\times C^1(\overline\Omega):\partial_\nu u=\partial_\nu v=0 \text{ on } \partial \Omega\}.$$
As in Eq. \eqref{egen}, here we still use $\mathbf{u}$ instead of $\mathbf{u}^T$ for simplicity.
Then system \eqref{stCM} is equivalent to
\begin{equation*}
\begin{cases}
  -D\Delta \mathbf u =G(\mathbf u), & x\in \Omega,\\
 \partial_\nu  \mathbf u=0,& x\in \partial \Omega,\\
\end{cases}
\end{equation*}
or
\begin{equation}\label{FF}
F(d_1,d_2,\mathbf{u})=\mathbf{u}-(I-\Delta)^{-1}\{D^{-1} G(\mathbf{u})+\mathbf{u}\}=\mathbf{0}\text{ on }\mathbf{X},
\end{equation}
where $D$ and $G(\mathbf{u})$ are defined as in \eqref{cgu}, $(I-\Delta)^{-1}$ is the inverse of $I-\Delta$ with the homogeneous Neumann boundary condition, and $\mathbf{0}=(u(x),v(x))\equiv(0,0)\in\mathbf{X}$. Let $\mathbf{u_i}=(u_i,v_i)\; (i=1,\cdots,n)$ be solutions of system \eqref{stCM}, where $n=1,2$ or $3$ under different conditions. As in \cite{NiW}, we also  define
$$
H_i(d_1,d_2,\la):=
d_1d_2\la^2+\left(d\beta \phi_2( v_i)d_1-d_2\phi_1(u_i)\psi'_1(u_i)\right)\la+\text{Det } G_{\mathbf u}(\mathbf{u_i}),
$$
where $G_{\mathbf u}(\mathbf{\cdot})$ is defined as in Eq. \eqref{trd}, and $\phi_i\;(i=1,2)$ and $\psi_1$ are defined in Eq. \eqref{pspi}. Actually,
$$H_i(d_1,d_2,\mu_j)=\text{Det } Q_j(\mathbf { u_i}),$$
where $\text{Det } Q_j(\mathbf{\cdot})$ is defined as in Eq. \eqref{trd}.
For any fixed $d_1$, $\alpha$, $\beta$, $c$, $d$, if $d_2$ is sufficiently large, then
$H_i(d_1,d_2,\la)=0$ has two real roots
\begin{equation*}
\begin{split}
&\lambda_i^{-}(d_1,d_2)=\ds\f{-P+\sqrt{P^2-4d_1d_2\text{Det } G_{\mathbf u}(\mathbf{u_i})}}{2d_1d_2},\\
&\lambda_i^{+}(d_1,d_2)=\ds\f{-P-\sqrt{P^2-4d_1d_2\text{Det } G_{\mathbf u}(\mathbf{u_i})}}{2d_1d_2},\\
\end{split}
\end{equation*}
where $$P=d\beta \phi_2( v_i)d_1-d_2\phi_1(u_i)\psi'_1(u_i).$$
Set
\begin{equation}\label{BS}
\begin{split}
&E=\{\mu_i:i\in\mathbb{N}_0\},\\
&B_i(d_1,d_2)=\{\la\ge0:\lambda_i^{-}(d_1,d_2)<\la<\lambda_i^{+}(d_1,d_2)\}.
\end{split}
\end{equation}
It follows from Lemma 5.2 of \cite{NiW}
that, if $H_i(d_1,d_2,\mu_j)\ne0$ for all $j\in\mathbb{N}_0$, then
\begin{equation}\label{inde}
\text{index }(F(d_1,d_2,\mathbf{\cdot}),\mathbf{u_i})=(-1)^{\gamma_i},
\end{equation}
where
\begin{equation}\label{gam}
\gamma_i=\begin{cases}
\sum_{\mu_j\in B_i\cap E}m(\mu_j) &B_i\cap E\ne\emptyset,\\
0 & B_i\cap E=\emptyset,
\end{cases}
\end{equation}
$m(\mu_j)$ is the multiplicity of $\mu_j$, and $F(d_1,d_2,\mathbf{\cdot})$ is defined as in Eq. \eqref{FF}. Similar to Section 3.1, the following discussion is also divided into two cases: case I and case II, which are defined as in Eqs. \eqref{ca2} and \eqref{ca1}. We first consider case II, and system \eqref{CM} may have three constant positive equilibria in this case.
\begin{theorem}
Assume that
$$\alpha>1,\;\ds\f{4\alpha}{(\alpha+1)^2}<\beta<1\text{ and }c>C(u_0),$$
where $u_0$ is defined as in Lemma \ref{l2}. Then the following two statements are true.
\begin{enumerate}
\item [(i)]
System \eqref{stCM} has three constant positive solutions $\mathbf{u_i}=(u_i,v_i)\;(i=1,2,3)$ satisfying $u_1<u_2<\ds\f{\alpha-1}{2\alpha}<u_3$.
\item[(ii)] If $\phi_1(u_1)\psi'_1(u_1)/d_1\in (\mu_p,\mu_{p+1})$ and $\phi_1(u_2)\psi'_1(u_2)/d_1\in (\mu_q,\mu_{q+1})$ for some $p\ge1$ and $q\ge1$, where $\phi_1$ and $\psi_1$ are defined as in Eq. \eqref{pspi}, and $\sum_{i=1}^p m(\mu_i)+\sum_{i=1}^q m(\mu_i)$ is odd, then there exists a positive constant $\hat d_2=\hat d_2(d_1,\alpha,\beta,d,c)$ such that system \eqref{stCM} has at least one nonconstant positive solution for any $d>\hat d_2$.
\end{enumerate}
\end{theorem}
\begin{proof}
Due to Lemma \ref{l2}, Theorems \ref{equ} and \ref{sig}, we see that system \eqref{stCM} has three positive constant solutions $\mathbf{u_i}=(u_i,v_i)\;(i=1,2,3)$ satisfying $u_1<u_2<\ds\f{\alpha-1}{2\alpha}<u_3$, and
 \begin{equation*}
 \begin{split}
 &\text{Det } G_{\mathbf u}(\mathbf{u_2})<0, \;\;\text{Det } G_{\mathbf u}(\mathbf{u_i})>0 \text{ for } i=1,3,\\
 &\psi'_1(u_3)<0, \;\; \psi'_1(u_i)>0 \text{ for }i=1,2.
 \end{split}
 \end{equation*}
Therefore,
\begin{equation} \label{limd}
\lim_{d_2\to\infty}\lambda_i^{+}(d_1,d_2)=\phi_1(u_i)\psi'_1(u_i)/d_1>0,\;\lim_{d_2\to\infty}\lambda_i^{-}(d_1,d_2)=0,\text{ for }i=1,2.
\end{equation}
Because $\phi_1(u_1)\psi'_1(u_1)/d_1\in (\mu_p,\mu_{p+1})$ and $\phi_1(u_2)\psi'_1(u_2)/d_1\in (\mu_q,\mu_{q+1})$,
by virtue of Eq. \eqref{limd}, we see that there exists $\hat d_2=\hat d_2(d_1,\alpha,\beta,d,c)$ such that, for all $d_2>\hat d_2$,
\begin{equation}\label{lae}
\begin{split}
&0<\lambda_1^{-}(d_1,d_2)<\mu_1,\;\;\mu_p<\lambda_1^{+}(d_1,d_2)<\mu_{p+1},\\
&\lambda_2^{-}(d_1,d_2)<0,\;\;\mu_q<\lambda_2^{+}(d_1,d_2)<\mu_{q+1}.
\end{split}
\end{equation}
It follows from Theorem \ref{coef} that there exists $d_*(\alpha,\beta,c,d,\Omega)$ such that, for all $d_1,d_2>d_*$, system \eqref{stCM} has no nonconstant positive steady states. We choose $\tilde d_1>d_*$ satisfies $$\phi_1(u_i)\psi'_1(u_i)/{\tilde d_1}<\mu_1\text{ for } i=1,2,$$
and hence we can choose $\tilde d_2>d_*$ satisfies
\begin{equation}\label{lae1}
0<\lambda_1^{-}(\tilde d_1,\tilde d_2)<\lambda_1^{+}(\tilde d_1,\tilde d_2)<\mu_1,\;\lambda_2^{-}(\tilde d_1,\tilde d_2)<0<\lambda_2^{+}(\tilde d_1,\tilde d_2)<\mu_{1}.
\end{equation}
Then we claim that system \eqref{stCM} has at least one nonconstant positive steady states for all $d_2>\hat d_2$. If this is not true, then there exists $d_2>\hat d_2$ such that system \eqref{stCM} has no nonconstant positive steady states. As in Theorem 5.7 of \cite{NiW}, we also define
\begin{equation}\label{cgu1}
D(t)=\left(\begin{array}{cc}
td_1+(1-t)\tilde d_1&0\\
0&td_2+(1-t)\tilde d_2
\end{array}\right),\;\;t\in[0,1],
\end{equation}
and
\begin{equation}\label{ph}\Phi(\mathbf{u},t)=\mathbf{u}-(I-\Delta)^{-1}\{D^{-1}(t) G(\mathbf{u})+\mathbf{u}\}=\mathbf{0}\text{ on }\mathbf{X}.
\end{equation}
Then, $$\Phi(\mathbf{u},1)=F(d_1,d_2,\mathbf{u})\;\;\text{and}\;\;\Phi(\mathbf{u},0)=F(\tilde d_1,\tilde d_2,\mathbf{u}).$$
By virtue of Eqs. \eqref{lae} and \eqref{lae1}, we have
\begin{equation}\label{inde1}
\begin{split}
&\text{index}(\Phi(\mathbf{\cdot},1),\mathbf{u_1})=\text{index}(F(d_1,d_2,\mathbf{\cdot}),\mathbf{u_1})=(-1)^{\sum_{i=1}^pm(\mu_i)},\\
&\text{index}(\Phi(\mathbf{\cdot},1),\mathbf{u_2})=\text{index}(F(d_1,d_2,\cdot),\mathbf{u_2})=(-1)^{\sum_{i=1}^qm(\mu_i)+1},\\
&\text{index}(\Phi(\mathbf{\cdot},0),\mathbf{u_1})=\text{index}(F(\tilde d_1,\tilde d_2,\mathbf{\cdot}),\mathbf{u_1})=1,\\
&\text{index}(\Phi(\mathbf{\cdot},0),\mathbf{u_2})=\text{index}(F(\tilde d_1,\tilde d_2,\mathbf{\cdot}),\mathbf{u_2})=-1.\\
\end{split}
\end{equation}
Moreover, noticing that $\text{Det } G_{\mathbf u}(\mathbf{u_3})>0$ and
$\psi'_1(u_3)<0$, we have $H_i(d_1,d_2,\la)>0$ for all $\la\ge0$ and $d_1,d_2>0$, which leads to
\begin{equation}\label{inde2}
\begin{split}
&\text{index}(\Phi(\mathbf{\cdot},1),\mathbf{u_3})=\text{index}(F(d_1,d_2,\mathbf{\cdot}),\mathbf{u_3})=1,\\
&\text{index}(\Phi(\mathbf{\cdot},0),\mathbf{u_3})=\text{index}(F(\tilde d_1,\tilde d_2,\mathbf{\cdot}),\mathbf{u_3})=1.
\end{split}
\end{equation}
It follows from Theorem \ref{les} that there exist two positive constants $\overline C$ and $\underline C$ such that,
for all $0\le t\le 1$, the positive solution $(u(x),v(x))$ of system \eqref{les} satisfies $$\ds\f{1}{2\underline C}<u(x),v(x)<2\overline C.$$
Here $\overline C$ and $\underline C$ are independent of $d_2$, and depend on $d_1$, $\tilde d_1$, $\tilde d_2$, $\hat d_2$, $\alpha$, $\beta$, $d$ and $c$.
Define $$M=\{\mathbf{u}=(u,v)\in\mathbf{X}:\ds\f{1}{2\underline C}<u(x),v(x)<2\overline C\}.$$
Then $\Phi(\mathbf{u},t)\ne\mathbf{0}$ for all $\mathbf{u}\in\partial M$ and $t\in[0,1]$, and by the Leray-Schauder degree theory, we have
$$\text{deg}(\Phi(\mathbf{\cdot},0),M,\mathbf{0})=\text{deg}(\Phi(\mathbf{\cdot},1),M,\mathbf{0}).$$
Then, taking advantage of Eqs. \eqref{inde1} and \eqref{inde2}, we obtain that
\begin{equation*}
\begin{split}
&\text{deg}(\Phi(\mathbf{\cdot},1),M,\mathbf{0})\\
=&\text{index}(\Phi(\mathbf{\cdot},1),\mathbf{u_1})+\text{index}(\Phi(\mathbf{\cdot},1),\mathbf{u_2})+\text{index}(\Phi(\mathbf{\cdot},1),\mathbf{u_3})\\
=&(-1)^{\sum_{i=1}^pm(\mu_i)}+(-1)^{\sum_{i=1}^qm(\mu_i)+1}+1=3\text{ or }-1,\\
&\text{deg}(\Phi(\mathbf{\cdot},0),M,\mathbf{0})\\
=&\text{index}(\Phi(\mathbf{\cdot},0),\mathbf{u_1})+\text{index}(\Phi(\mathbf{\cdot},0),\mathbf{u_2})+\text{index}(\Phi(\mathbf{\cdot},0),\mathbf{u_3})=1,\\
\end{split}
\end{equation*}
which is a contradiction. This completes the proof.
\end{proof}
Similarly, we can derive the following three results for case I. Here we omit the proof.
\begin{theorem}
Assume that
\begin{equation*}
\alpha>1,\;\gamma(\alpha)<\beta<\ds\f{4\alpha}{(\alpha+1)^2},\;\text{and}\;c>C(u_2),
\end{equation*}
where $\gamma(\alpha)$ and $u_2$ are defined as in Lemma \ref{l2}. Then the following two statements are true.
\begin{enumerate}
\item [(i)]
System \eqref{stCM} has a unique constant positive solution $\mathbf{u_1}=(u_1,v_1)$ satisfying $u_1<\ds\f{\alpha-1}{2\alpha}$.
\item[(ii)] If $\phi_1(u_1)\psi'_1(u_1)/d_1\in (\mu_p,\mu_{p+1})$ for some $p\ge1$, where $\phi_1$ and $\psi_1$ are defined as in Eq. \eqref{pspi}, and $\sum_{i=1}^p m(\mu_i)$ is odd, then there exists a positive constant $\hat d_2=\hat d_2(d_1,\alpha,\beta,d,c)$ such that system \eqref{stCM} has at least one nonconstant positive solution for any $d>\hat d_2$.
\end{enumerate}
\end{theorem}
\begin{theorem}
Assume that
\begin{equation*}
\alpha>1,\;\gamma(\alpha)<\beta<\ds\f{4\alpha}{(\alpha+1)^2},\;\text{and}\;\max\left\{C(u_1),C\left(\ds\f{\alpha-1}{2\alpha}\right)\right\}<c<C(u_2),
\end{equation*}
where $\gamma(\alpha)$, $u_1$ and $u_2$ are defined as in Lemma \ref{l2}. Then the following two statements are true.
\begin{enumerate}
\item [(i)]
System \eqref{stCM} has three positive constant solutions $\mathbf{u_i}=(u_i,v_i)\;(i=1,2,3)$ satisfying $u_1<u_2<u_3<\ds\f{\alpha-1}{2\alpha}$.
\item[(ii)] If
\begin{equation*}
\begin{split}
&\phi_1(u_1)\psi'_1(u_1)/d_1\in (\mu_p,\mu_{p+1}),\\
&\phi_1(u_2)\psi'_1(u_2)/d_1\in (\mu_q,\mu_{q+1}),\\
&\phi_1(u_3)\psi'_1(u_3)/d_1\in (\mu_r,\mu_{r+1}),
\end{split}
\end{equation*}
for some $p\ge1$, $q\ge1$ and $r\ge1$, where $\phi_1$ and $\psi_1$ are defined as in Eq. \eqref{pspi}, and $$\sum_{i=1}^p m(\mu_i)+\sum_{i=1}^q m(\mu_i)+\sum_{i=1}^r m(\mu_i)$$ is odd, then there exists a positive constant $\hat d_2=\hat d_2(d_1,\alpha,\beta,d,c)$ such that system \eqref{stCM} has at least one nonconstant positive solution for any $d>\hat d_2$.
\end{enumerate}
\end{theorem}
\begin{theorem}
Assume that
\begin{equation*}
\alpha>1,\; \beta <\gamma(\alpha) ,\;\text{and}\;c>C\left(\ds\f{\alpha-1}{2\alpha}\right),
\end{equation*}
where $\gamma(\alpha)$ are defined as in Lemma \ref{l2}. Then the following two statements are true.
\begin{enumerate}
\item [(i)]
System \eqref{stCM} has a unique constant solution $\mathbf{u_1}=(u_1,v_1)$ satisfying $u_1<\ds\f{\alpha-1}{2\alpha}$.
\item[(ii)] If $\phi_1(u_1)\psi'_1(u_1)/d_1\in (\mu_p,\mu_{p+1})$ for some $p\ge1$, where $\phi_1$ and $\psi_1$ are defined as in Eq. \eqref{pspi}, and $\sum_{i=1}^p m(\mu_i)$ is odd, then there exists a positive constant $\hat d_2=\hat d_2(d_1,\alpha,\beta,d,c)$ such that system \eqref{stCM} has at least one nonconstant positive solution for any $d>\hat d_2$.
\end{enumerate}
\end{theorem}
\section{Conclusions}
This paper mainly deals with the effects of the predator inference $\beta$ and conversion rate of the predator $c$ on a diffusive predator-prey model. we see that if the conversion rate is small, or the predator inference is strong, then the dynamics of system \eqref{CM} is simple, and all the solutions, regardless of the initial dates, converge to a constant steady state as time goes to infinity.
However, if the predator inference is neither strong nor weak, then system \eqref{CM} may have multiple constant positive equilibria and hence the dynamics is complex. We also find that there exist no nonconstant positive steady states when the conversion rate of the predator is large, and nonconstant positive steady states emerge when the diffusion rate of the predator is large.

We also remark that the results and methods used here cannot only be applied to CM functional response, but also for other functional responses with predator interference. For example, we can similarly obtain that, the following predator-prey model with BD functional response,
\begin{equation}\label{CM1s}
\begin{cases}
  \ds\frac{\partial u}{\partial t}-d_1\Delta u=ru\left(1-\ds\f{u}{k}\right)-\ds\frac{buv}{1+\alpha u+\beta v}, & x\in \Omega,\; t>0,\\
 \ds\frac{\partial v}{\partial t}-d_2\Delta v=-dv+\ds\frac{cuv}{1+\alpha u+\beta v}, & x\in\Omega,\; t>0,\\
 \partial_\nu  u=\partial_\nu
  v=0,& x\in \partial \Omega,\;
 t>0,\\
\end{cases}
\end{equation}
has no nonconstant positive steady states when the conversion rate of the predator is large.


\begin{thebibliography}{10}

\bibitem{Bazykin}
A.~D. Bazykin.
\newblock {\em Nonlinear {D}ynamics of {I}nteracting {P}opulations}.
\newblock World Scientific, 1988.

\bibitem{Beddington}
J.~R. Beddington.
\newblock Mutual interference between parasites or predators and its effect on
  searching efficiency.
\newblock {\em J. Animal Ecol.}, 44:331--340, 1975.

\bibitem{Berec}
L.~Berec.
\newblock Impacts of foraging facilitation among predators on predator-prey
  dynamics.
\newblock {\em Bull. Math. Biol.}, 72:94--121, 2010.

\bibitem{Cantrell}
R.~S. Cantrell and C.~Cosner.
\newblock On the dynamics of predator-prey models with the
  {B}eddington-{D}e{A}ngelis functional response.
\newblock {\em J. Math. Anal. Appl.}, 257:206--222, 2001.

\bibitem{Chen-yu2}
S.~Chen and J.~Yu.
\newblock Dynamics of a diffusive predator-prey system with a nonlinear growth
  rate for the predator.
\newblock {\em J. Differential Equations}, 260(11):7923--7939, 2016.

\bibitem{Cheng}
K.-S. Cheng.
\newblock Uniqueness of a limit cycle for a predator-prey system.
\newblock {\em SIAM J. Math. Anal.}, 12(4):541--548, 1981.

\bibitem{Crowley}
P.H. Crowley and E.K. Martin.
\newblock Functional responses and interference within and between year classes
  of a dragonfly population.
\newblock {\em J. North Am. Benthol. Soc.}, 8(3):211--221, 1989.

\bibitem{DeAngelis}
D.~L. DeAngelis, R.~A. Goldstein, and R.~V. O'Neill.
\newblock A model for trophic interaction.
\newblock {\em Ecology}, 56:881--892, 1975.

\bibitem{Du-HSU2}
Y.~Du and S.-B. Hsu.
\newblock On a nonlocal reaction-diffusion problem arising from the modeling of
  phytoplankton growth.
\newblock {\em SIAM J. Math. Anal.}, 42(3):1305--1333, 2010.

\bibitem{Du-Lou3}
Y.~Du and Y.~Lou.
\newblock Some uniqueness and exact multiplicity results for a predator-prey
  model.
\newblock {\em Trans. Amer. Math. Soc.}, 349(6):2443--2475, 1997.

\bibitem{Du-Lou}
Y.~Du and Y.~Lou.
\newblock Qualitative behaviour of positive solutions of a predator-prey model:
  effects of saturation.
\newblock {\em Proc. Roy. Soc. Edinburgh Sect. A}, 131(2):321--349, 2001.

\bibitem{Du-Shi1}
Y.~Du and J.~Shi.
\newblock A diffusive predator-prey model with a protection zone.
\newblock {\em J. Differential Equations}, 229:63--91, 2006.

\bibitem{Du-Shi}
Y.~Du and J.~Shi.
\newblock Allee effect and bistability in a spatially heterogeneous
  predator-prey model.
\newblock {\em Trans. Amer. Math. Soc.}, 359(9):4557--4593, 2007.

\bibitem{Holling}
C.~S. Holling.
\newblock Some characteristics of simple types of predation and parasitism.
\newblock {\em Can. Entomol.}, 91(7):385--398, 1959.

\bibitem{Hsu1}
S.-B. Hsu.
\newblock On global stability of a predator-prey system.
\newblock {\em Math. Biosci.}, 39:1--10, 1978.

\bibitem{Hsu2}
S.-B. Hsu, S.P. Hubbell, and P.~Waltman.
\newblock Competing predators.
\newblock {\em SIAM J. Appl. Math.}, 35(4):617--625, 1978.

\bibitem{Hsu3}
S.-B. Hsu and S.~Shi.
\newblock Relaxation oscillation profile of limit cycle in predator-prey
  system.
\newblock {\em Discrete Contin. Dyn. Syst. Ser. B}, 11(4):893--911, 2009.

\bibitem{LiW}
S.~Li and J.~Wu amd Y.~Dong.
\newblock Uniqueness and stability of a predator-prey model with {C}-{M}
  functional response.
\newblock {\em Comput. Math. Appl.}, 69(10):1080--1095, 2015.

\bibitem{Lieberman}
G.~M. Lieberman.
\newblock Bounds for the steady-state {S}el'kov model for arbitrary $p$ in any
  number of dimensions.
\newblock {\em SIAM J. Math. Anal.}, 36(5):1400--1406, 2005.

\bibitem{Lin-Ni}
C.-S. Lin, W.-M. Ni, and I.~Takagi.
\newblock Large amplitude stationary solutions to a chemotaxis systems.
\newblock {\em J. Differential Equations}, 72:1--27, 1988.

\bibitem{Lou-Ni}
Y.~Lou and W.-M. Ni.
\newblock Diffusion, self-diffusion and cross-diffusion.
\newblock {\em J. Differential Equations}, 131(1):79--131, 1996.

\bibitem{Murdoch}
W.~W. Murdoch, C.~J. Briggs, and R.~M. Nisbet.
\newblock {\em Consumer-{R}esource {D}ynamics}.
\newblock Princeton University Press, 2003.

\bibitem{Murray}
J.~D. Murray.
\newblock {\em Mathematical {B}iology}.
\newblock Springer-Verlag, 2002.

\bibitem{NiW}
W.~Ni and M.~Wang.
\newblock Dynamics and patterns of a diffusive {L}eslie-{G}ower prey-predator
  model with strong {A}llee effect in prey.
\newblock {\em J. Differential Equations}, 261:4244--4274, 2016.

\bibitem{pangw}
P.~Y.~H. Pang and M.~Wang.
\newblock Non-constant positive steady states of a predator-prey system with
  non-monotonic functional response and diffusion.
\newblock {\em Proc. London Math. Soc.}, 88(3):135--157, 2004.

\bibitem{Peng-Shi}
R.~Peng and J.~Shi.
\newblock Non-existence of non-constant positive steady states of two {H}olling
  type-{II} predator-prey systems: Strong interaction case.
\newblock {\em J. Differential Equations}, 247(3):866--886, 2009.

\bibitem{Peng-Shi-Wang2}
R.~Peng, J.~Shi, and M.~Wang.
\newblock On stationary patterns of a reaction-diffusion model with
  autocatalysis and saturation law.
\newblock {\em Nonlinearity}, 21(7):1471--1488, 2008.

\bibitem{RuanX1}
S.~Ruan and D.~Xiao.
\newblock Global analysis in a predator-prey system with nonmonotonic
  functional response.
\newblock {\em SIAM J. Appl. Math.}, 61(4):1445--1472, 2001.

\bibitem{Sambath}
M.~Sambath, S.~Gnanavel, and K.~Balachandran.
\newblock Stability and {H}opf bifurcation of a diffusive predator-prey model
  with predator saturation and competition.
\newblock {\em Appl. Anal.}, 92(12):2451--2468, 2013.

\bibitem{Seo2}
G.~Seo and D.~L. DeAngelis.
\newblock A predator-prey model with a {H}olling type {I} functional response
  including a predator mutual interference.
\newblock {\em J. Nonlinear Sci.}, 21:811--833, 2011.

\bibitem{Seo}
G.~Seo and M.~Kot.
\newblock A comparison of two predator-prey models with {H}olling's type {I}
  functional response.
\newblock {\em Math. Biosci.}, 212:161--179, 2008.

\bibitem{ShiR}
H.-B. Shi and S.~G. Ruan.
\newblock Spatial, temporal and spatiotemporal patterns of diffusive
  predator-prey models with mutual interference.
\newblock {\em IMA. J. Appl. Math.}, 80(5):1534--1568, 2015.

\bibitem{WangJF}
J.~Wang.
\newblock Spatiotemporal patterns of a homogeneous diffusive predator-prey
  system with {H}olling type {I}{I}{I} functional response.
\newblock {\em To appear in J. Dyn. Diff. Equat.}, DOI:
  10.1007/s10884-016-9517-7.

\bibitem{WangSW}
J.~Wang, J.~Shi, and J.~Wei.
\newblock Dynamics and pattern formation in a diffusive predator-prey system
  with strong {A}llee effect in prey.
\newblock {\em J. Differential Equations}, 251(4-5):1276--1304, 2011.

\bibitem{Wang}
J.~Wang, J.~Wei, and J.~Shi.
\newblock Global bifurcation analysis and pattern formation in homogeneous
  diffusive predator-prey systems.
\newblock {\em J. Differential Equations}, 260:3495--3523, 2016.

\bibitem{WangM}
M.~Wang and Q.~Wu.
\newblock Positive solutions of a predator-prey model with predator saturation
  and competition.
\newblock {\em J. Math. Anal. Appl.}, 345(2):708--718, 2008.

\bibitem{WeiW}
M.~Wei, J.~Wu, and G.~Guo.
\newblock The effect of predator competition on positive solutions for a
  predator-prey model with diffusion.
\newblock {\em Nonlinear Anal.}, 75(13):5053--5068, 2012.

\bibitem{RuanX2}
D.~Xiao and H.~Zhu.
\newblock Multiple focus and {H}opf bifurcations in a predator-prey system with
  nonmonotonic functional response.
\newblock {\em SIAM J. Appl. Math.}, 66(3):802--819, 2006.

\bibitem{Yi}
F.~Yi, J.~Wei, and J.~Shi.
\newblock Bifurcation and spatiotemporal patterns in a homogeneous diffusive
  predator-prey system.
\newblock {\em J. Differential Equations}, 246(5):1944--1977, 2009.

\bibitem{ZhangJ}
X.-C. Zhang, G.-Q. Sun, and Z.~Jin.
\newblock Spatial dynamics in a predator-prey model with
  {B}eddington-{D}e{A}ngelis functional response.
\newblock {\em Physical Review E}, 85:021924, 2012.

\bibitem{ZhouJ2}
J.~Zhou.
\newblock Positive solutions of a diffusive {L}eslie-{G}ower predator-prey
  model with {B}azykin functional response.
\newblock {\em Z. Angew. Math. Phys.}, 65(1):1--18, 2014.

\bibitem{ZhouJ1}
J.~Zhou.
\newblock Qualitative analysis of a modified {L}eslie-{G}ower predator-prey
  model with {G}rowley-{M}artin funtional responses.
\newblock {\em Commun. Pur. Appl. Anal.}, 14(3):1127--1145, 2015.

\bibitem{muzhou}
J.~Zhou and C.~Mu.
\newblock Coexistence states of a {H}olling type-{I}{I} predator-prey system.
\newblock {\em J. Math. Anal. Appl.}, 369(2):555--563, 2010.

\bibitem{ZhuC}
H.~Zhu, S.~A. Campbell, and G.~S.~K. Wolkowicz.
\newblock Bifurcation analysis of a predator-prey system with nonmonotonic
  functional response.
\newblock {\em SIAM J. Appl. Math.}, 63(2):636--682, 2002.

\end{thebibliography}
\end{document}